\documentclass[twoside]{amsart}

\numberwithin{equation}{section}
\usepackage{amsmath,amssymb,amsfonts,amsthm,latexsym}
\usepackage{graphicx,color,enumerate}
\usepackage[margin=1in]{geometry}
\usepackage{longtable}
\usepackage{appendix}
\usepackage{xcolor}
\usepackage{stmaryrd}
\usepackage{mleftright}
\usepackage{breqn}
\usepackage{url}
\usepackage{orcidlink}
\usepackage{epigraph}

\newtheorem{theorem}{Theorem}[section]
\newtheorem{lemma}[theorem]{Lemma}
\newtheorem{proposition}[theorem]{Proposition}
\newtheorem{corollary}[theorem]{Corollary}

\theoremstyle{definition}
\newtheorem{definition}[theorem]{Definition}

\theoremstyle{remark}
\newtheorem{remark}[theorem]{Remark}

\begin{document}

\title{An exploration of connections and curvature in the presence of singularities}

\author[H.-C.~Herbig]{Hans-Christian Herbig\,\orcidlink{0000-0003-2676-3340}}
\address{Departamento de Matem\'{a}tica Aplicada, Universidade Federal do Rio de Janeiro,
Av. Athos da Silveira Ramos 149, Centro de Tecnologia - Bloco C, CEP: 21941-909 - Rio de Janeiro, Brazil}
\email{herbighc@gmail.com}

\author[W. Osnayder Clavijo E.]{
William Osnayder Clavijo Esquivel\,\orcidlink{0009-0000-4211-3300
}}
\address{Departamento de Matem\'{a}tica Aplicada, Universidade Federal do Rio de Janeiro,
Av. Athos da Silveira Ramos 149, Centro de Tecnologia - Bloco C, CEP: 21941-909 - Rio de Janeiro, Brazil}
\email{woclavijo@im.ufrj.br}


\begin{abstract} 
We develop the notions of connections and curvature for general Lie-Rinehart algebras without using smoothness assumptions on the base space. We present situations when a connection exists. E.g., this is the case when the underlying module is finitely generated. We show how the group of module automorphism acts as gauge transformations on the space of connections. When the underlying module is projective we define a version of the Chern character reproducing results of Hideki Ozeki. We discuss various examples of flat connections and the associated Maurer-Cartan equations. We provide examples of Levi-Civita connections on singular varieties and singular differential spaces with non-zero Riemannian curvature. The main observation is that for quotient singularities, even though the metric degenerates along strata, the poles of the Christoffel symbols are removable.
\end{abstract}

\maketitle

\tableofcontents
\section{Introduction}\label{sec:intro}

In the past we discovered flat connections on singular varieties that arise naturally in Poisson geometry \cite{higherKoszul,Nambuffel} and generalizations thereof \cite{Nambuffel} 
for Nambu-Poisson geometry (see Section \ref{sec:flatex}). The question of the existence of Levi-Civita connections on singular spaces also arises when one attempts to extend Boris Dubrovin's construction \cite{Dubrovin} of Frobenius manifolds from the principal stratum of the orbit space of a Coxeter group to the whole orbit space. It turns out that the usual theory of differential geometry can be developed without smoothness assumptions on the base space and without
supposing that the module where the connection resides is projective. However, it is not clear a priori whether the set of connections is nonempty. To our knowledge the corresponding obstructions have still to be worked out.

In this paper we present a basic theory of connections and curvature without putting any assumptions. The framework is based on the notion of a Lie-Rinehart algebra $ ( L,A)$ and on what we call its naive de Rham complex (see \cite{Rinehart}). This idea is actually quite old: in \cite{Ozeki} Hideki Ozeki presented such an approach with the objective to construct a Chern character for the case when the $ A$-module $ V$ where the connection lives is a finitely generated projective $ A$-module. We are able to effortlessly reproduce his results (see Section \ref{sec:Chern}). Due to our ignorance we do not know if something similar can be done when the projectivity assumption on $ V$ is dropped. The problem to start with is to find regular functions on $ \operatorname{End}_{A}( V)$ that are invariant under conjugation with elements of $ \operatorname{Aut}_{A}( V)$ and which could potentially replace the trace polynomials. Even if $ V$ is finitely generated of projective dimension $ 1$ we do not know how to construct such $ \operatorname{Aut}_{A}( V)$-invariant regular functions. The group of $A$-module automorphisms $ \operatorname{Aut}_{A}( V)$ is of central  importance since it can be understood as the group of gauge transformations acting on the space $ \operatorname{Conn}_{L}( V)$ of $ L$-connections on $ V$. Our gauge transformations generalize the constructions that are well-known for the case when $A$ is smooth and $L$ and $V$ are finitely generated projective $A$-modules. The realization that a gauge action of $ \operatorname{Aut}_{A}( V)$ on 
$ \operatorname{Conn}_{L}( V)$ can be defined without any assumption on $A,\ L$ or $ V$ came as a surprise to us (see Section \ref{sec:gauge}).

It is actually not difficult to find examples of flat connections on singular spaces (see Section
\ref{sec:flatex}). However, in view of the usual theory of Chern classes it seems desirable to find examples of non-flat connections. At the moment we are not able to present non-flat examples of $A$-modules $V$ that do not admit flat connections which go beyond the well-known theory for the smooth case. The task to find non-flat examples is only straight forward in the case when $ V$ is a finitely generated projective module (see Section \ref{subsec:Fedosov}). The attempts to construct Levi-Civita connections on singular varieties using the embedding metric run into problems. It appears to be necessary to localize along the singularities (see Section \ref{sec:LC}). However, when the singularities in question stem from taking categorical quotients of reductive (or compact) symmetry groups then there is hope, since the problems are caused by singularities that occur just by declaring points in the same (closed) orbits to be equivalent. As equivalence relations are in a sense just some form of mental gymnastics one should expect that the technical problems can be overcome. In fact, we are able to give examples of non-flat Levi-Civita connections on such categorical quotients and note that those can be found quite generally using the construction principles employed here (see Section \ref{sec:LC}). We emphasize that Riemannian metrics on orbit spaces of compact groups, viewed as stratified spaces, have been studied before \cite{Michor}. To our knowledge a notion of a Levi-Civita connection has not yet been elaborated in this approach.
We present the Levi-Civita connection on a simple example of an orbit space, namely the quotient of the real Coxeter group $A_2$
(see Section \ref{sec:Dubrovin}).

As already said, we use the language of Lie-Rinehart algebras, their modules and their cohomologies as the basic tool to talk about connections \cite{Rinehart}. This has, among other things, the advantage that the approach is flexible. We are not obliged to stay within the framework of complex or real algebraic geometry. A framework for singular spaces that is more adequate for the needs of physics is provided by differential spaces (see \cite{Navarro, Sniatycki,FHS}). It is also possible to work in an analytic or subanalytic setup. If the underlying $A$-module $L$ of a Lie-Rinehart algebra $(L,A)$ is not projective there are technical challenges that require attention. For many natural examples of $L$, such as the module of derivations of a singular algebra $A$, the $A$-module $L$ is not reflexive. In this case the cohomology of a Lie-Rinehart algebra, which we refer in this paper as the naive de Rham cohomology, cannot be determined by deriving a Hom functor. In the framework of complex algebraic geometry the naive de Rham complex is known as the de Rham complex of reflexive forms and can be calculated using resolutions of singularities (see, e.g., \cite{Greb}). We do not know of any quantitative results for differential spaces (see, however, \cite{Watts}). Recently, the authors used the naive de Rham complex to construct symplectic forms on singular Poisson algebras \cite{HOS}.

When $A$ is a noetherian $\boldsymbol{k}$-algebra one says that it is smooth if and only if its $A$-module of Kähler differentials $\Omega_{A|\boldsymbol{k}}$ is projective. This is a succinct and useful terminology since it coincides with what intuitively meant by a singularity. We do not know of any non-smooth noetherian $\boldsymbol{k}$-algebra with reflexive module of Kähler differentials  $\Omega_{A|\boldsymbol{k}}$. We are not aware, however, of general results that claim that such examples do not exist. On the other hand, if $A$ is the algebra of smooth functions on  a differential space there are no homological smoothness criteria in terms of the module of smooth differential forms (see \cite{Navarro}). The assessment whether the algebra of smooth functions on a manifold with boundary is to be considered smooth appears to be a matter of taste. We adopt the point of view that the algebra $A$ of functions on a differential space is to be considered \emph{smooth} if and only if the differential space is diffeomorphic to a smooth manifold (i.e., manifolds with boundaries are not considered to be smooth here) and otherwise \emph{singular}. In Section \ref{sec:Dubrovin} we study an example of a Levi-Civita connection on an orbit space of 
coregular real group action. Via the Hilbert embedding such spaces are viewed as semialgebraic sets that are merely defined in terms of polynomial inequalities. It follows that in this situation the $A$-module of smooth forms (cf. \cite{Navarro}) is free and hence reflexive.

We confess that we feel that there is no demand for interdisciplinary (i.e., bilingual) papers like ours. We found it almost impossible to tag the article properly with a mathematical subject classification.
A differential geometer or physicist may criticise that we are mostly concerned with primordial slime that can be found in text books. For a generic commutative algebraist the geometric motivations that guide our search are unappealing. We emphasize that the mathematization of gauge theory\footnote{This is a mayor endeavour of theoretical physics that started about hundred years ago, when people became interested in the geometry underlying Maxwell's theory of electromagnetism.} faces two serious challenges: infinite dimensionality and quotient singularities arising from dividing out the action of the gauge group. For this reason we feel it to be necessary to extend the geometric formalism to the realm of quotient singularities using a clear, conceptual language that is not burdened with unnecessary technicalities. The use of upper indices appears to be revolting to the algebraist, but the geometer's brain, wired in Ricci calculus, cannot work cleanly without it. Also the inconvenience that in our work many words (e.g., flatness, torsion-freeness etc.) have a double meaning is stemming from the fact that algebraists and differential geometers usually mind their own respective businesses. We hope that despite of this our paper is readable to either group of people. Since what is clear to one part of the readership may be obscure to the other we present a lot of details. Note, however, that Sections \ref{sec:LC} and \ref{sec:Dubrovin} require familiarity with polynomial and smooth invariant theory (see \cite{Bierstone, liftingHomo, HOS}).

Let us give a brief outline of the paper. In Section \ref{sec:LR} we recall the notion of a Lie-Rinehart algebra and its naive de Rham complex. In Section \ref{sec:Lconn} we define what we mean by an $L$-connection and discuss basic properties and examples. In Section \ref{sec:curv} we define curvature and torsion and exhibit a canonical flat connection on an endomorphism module. In Section \ref{sec:gauge} we define the action of the gauge group $\operatorname{Aut}_A(V)$ on the set $\operatorname{Conn}_L(V)$ of $L$-connections on the $A$-module $V$. In Section \ref{sec:Chern} we construct a Chern character for the special case when $V$ is a projective module. In Section \ref{sec:flatex} we present simple examples of flat $L$-connections on singular varieties and in Section \ref{sec:LRmod} we explain the equivalence of the notion of a flat $L$-connection on an $A$-module $V$ and a Lie-Rinehart module structure on $V$. In Section \ref{sec:Fed} we recall a simple construction of connections on projective modules. In Section \ref{sec:LC}
we present Levi-Civita connections and Riemannian curvature on the double cone, while in Section \ref{sec:Dubrovin} we do the same
for the orbit space of the Coxeter group $A_2$, seen as a differential space in the sense of Sikorski. Finally, in Section \ref{sec:outlook} we indicate how the topics touched upon in our work can be further developed.

\vspace{2mm}
\noindent\textit{Acknowledgements.} This work grew out of earlier stages of a joint project with Christopher Seaton and Daniel Herden on higher Poisson cohomology. Our interest in Levi-Civita connections on categorical quotients and orbit spaces arose in discussions with Yassir Dinar on Boris Dubrovin's constructions of Frobenius manifolds. W.O.C.E. thanks Leandro Gustavo Gomes for introducing him to computer assisted curvature calculations and acknowledges financial support of CAPES. Decades ago HCH watched Martin Bordemann filling seemingly unlimited amounts of paper with calculations of Atiyah-Molino classes and coffee stains, but does not recall a punch line.

 \section{Lie-Rinehart algebras and the naive de Rham complex}\label{sec:LR}

 In this section we fix basic notions and conventions.
 Throughout this paper $\boldsymbol{k}$ denotes a field of characteristic zero. If not otherwise stated by a $\boldsymbol{k}$-algebra we mean a commutative associative $\boldsymbol{k}$-algebra with unit. 
If $A$ is a $\boldsymbol{k}$-algebra we denote by $\operatorname{Der}( A)$ the $A$-module of $\boldsymbol{k}$-linear derivations $X:A\rightarrow A$. 
It is a $\boldsymbol{k}$-Lie algebra with bracket $[ \ ,\ ] :\operatorname{Der}( A) \times \operatorname{Der}( A)\rightarrow \operatorname{Der}( A) ,\ ( X,Y) \mapsto X\circ Y-Y\circ X$.
More specifically, the $A$-module of derivations is an example of a \textit{Lie-Rinehart algebra}, 
a structure introduced by Richard Palais \cite{Palais} in 1961 under the name $d$-Lie ring and studied more systematically by George Stewart Rinehart \cite{Rinehart}. By a Lie-Rinehart algebra we mean a pair $( L,A)$ where $A$ is a $\boldsymbol{k}$-algebra, $L$ is an $A$-module that is supplied with a $\boldsymbol{k}$-bilinear Lie bracket $[ \ ,\ ] :L\times L\rightarrow L$ such that the following conditions hold:
\begin{enumerate}
\item  $L$ acts on $A$ by derivations, i.e., for $(X, a)\in L\times A$ is mapped to $X( a) \in A$ and $X( ab) =aX( b) +bX( a)$ \ for $a,b\in A$.
\item For all $X\in L,\ a,b\in A$ the relation $a( X( b)) =( aX)( b)$ is satisfied.
\item For all $X,Y\in L,\ a\in A$ there is the relation $[ X,aY] =X(a)Y +a[X,Y]$.
\end{enumerate}

As already said, for each $\boldsymbol{k}$-algebra $A$ there is the example of the so-called \textit{tangent Lie-Rinehart algebra} $\left(\operatorname{Der}( A) ,A\right)$. It is implicit in the definition that there is a canonical map $\alpha :L\rightarrow \operatorname{Der}( A)$, referred to as the \textit{anchor}, attached to each Lie-Rinehart algebra $( L,A)$. The anchor is an example of a morphism of Lie-Rinehart algebras. In fact, by a \textit{morphism between Lie-Rinehart algebras} $( L,A)$ and $( L',A)$ we mean a morphism of $A$-modules $L\rightarrow L'$ that is compatible with brackets and anchors. More generally, let $ \phi :A\rightarrow A'$ be a surjective $ \boldsymbol{k}$-algebra morphism.
Then a morphism of Lie-Rinehart algebras $ ( L,A)$ and $ ( L',A')$ is by definition a morphism $ \Phi :L\rightarrow L'$ of $ A$-modules compatible with the brackets such that $ \Phi ( X)( a') =\phi ( X( a))$ \ for all $ X\in L,\ a'\in A'\ $ independent of the choice of $ a\in \phi ^{-1}( a')$. Here $L'$ is understood to be an $A$-module via $\phi$.

The cohomology of a Lie-Rinehart algebra has been defined by Richard Palais \cite{Palais}. George S. Rinehart \cite{Rinehart} showed that in the case when $L$ is a projective $A$-module it can be defined by deriving a Hom functor. In this paper, we do not impose restrictions on the $A$-module $L$ we refer to this complex as the naive de Rham complex. The space of naive de Rham $m$-cochains, $m\geq 0$, of the Lie-Rinehart algebra $( L,A)$ is defined to be the $A$-module $\operatorname{Alt}_{A}^{m}( L,A)$ of alternating $A$-multilinear forms of arity $m$ from $L$ to $A$. The naive de Rham differential $d_{\operatorname{dR}} :\operatorname{Alt}_{A}^{m}( L,A)\rightarrow \operatorname{Alt}_{A}^{m+1}( L,A)$ is given by the well-known formula
\begin{align*}
 & (d_{\operatorname{dR}} \omega ) (X_{0} ,X_{1} ,\dotsc ,X_{m} )\\
 & =\sum _{l=0}^{m} (-1)^{l} X_{l} \left(\omega ( X_{0} ,\dotsc ,\widehat{X_{l}} ,\dotsc ,X_{m})\right)\\
 & \ \ \ \ \ \ \ \ \ \ \ \ \ \ \ \ \ \ \ \ \ \ \ \ \ \ \ \ +\sum _{0\leq k< l\leq m} (-1)^{k+l} \omega \left( [X_{k} ,X_{l} ],X_{0} ,\dotsc ,\widehat{X_{k}} ,\dotsc ,\widehat{X_{l}} ,\dotsc ,X_{m}\right) ,
\end{align*}
which is attributed to Jean-Louis Koszul. Here we assume $X_{0} ,\dotsc ,X_{m} \in L$ and the hats indicate omission of the corresponding term. There is a supercommutative product
$\cup :\operatorname{Alt}_{A}^{k}( L,A) \times \operatorname{Alt}_{A}^{m}( L,A)\rightarrow \operatorname{Alt}_{A}^{k+m}( L,A)$ defined by
$$\omega \cup \omega '( X_{1} ,\dotsc ,X_{k+m} )=\sum _{\sigma \in\operatorname{Sh}_{k,m}}( -1)^{|\sigma |} \omega ( X_{\sigma ( 1)} ,\dotsc ,X_{\sigma ( k)} )\omega ( X_{\sigma ( k+1)} ,\dotsc ,X_{\sigma ( k+m)} )$$
for $X_{1} ,\dotsc ,X_{k+m} \in L$.
Here $\operatorname{Sh}_{k,m}$ denotes the set of shuffle permutations of the set $\{1,\dots,k\}\sqcup\{k+1,\dots,m\}$.
Palais \cite{Palais} showed that $\left(\operatorname{Alt}_{A}^{m}( L,A) ,d_{\operatorname{dR}} ,\cup \right)$ forms a supercommutative dg algebra over $\boldsymbol{k}$. For the cohomology algebra we write $(\mathrm{H}_{\operatorname{dR}}( L,A) ,\cup )$. For the cohomology of Lie-Rinehart modules, which also goes back to \cite{Palais}, see Section \ref{sec:LRmod}.

\section{$L$-connection on an $A$-module $V$}\label{sec:Lconn}

Let $( L,A)$ be a Lie-Rinehart algebra over $\boldsymbol{k}$. By an $L$-connection on the $A$-module $V$ we mean an $A$-linear map $\nabla :V\rightarrow \operatorname{Hom}_{A}( L,A) \otimes _{\boldsymbol{k}} V$ such that $\nabla ( av) =d_{\operatorname{dR}} a\otimes v+a\nabla v$. Given a connection $\nabla $ and a $X\in L$ there is the evaluation map
$$\nabla _{X} :V\xrightarrow{\nabla }\operatorname{Hom}_{A}( L,A) \otimes _{\boldsymbol{k}} V\xrightarrow{\omega \otimes v\mapsto \omega ( X) v} V.$$ 
It is clear that $\nabla _{X} v=X( a) \otimes v+a\nabla _{X} v$ for all $X\in L,\ v\in V,\ a\in A$. Conversely, an $L$-connection is nothing but an $A$-linear assingment $X\mapsto \nabla _{X}$ satisfying this identity.

If $\nabla $ and $\nabla '$ are two $L$-connections on the $A$-module $V$ and $X\in L$ then $\nabla _{X} -\nabla '_{X} \in \operatorname{Hom}_{A}( L,A) \otimes _{A} V$ \ and the map $X\mapsto \nabla _{X} -\nabla '_{X}$ is $A$-linear.
That is, the difference $\nabla -\nabla ' \in \operatorname{Hom}_{A}( L,A) \otimes _{A}\operatorname{End}_{A} (V)$ can be interpreted as an $\operatorname{End}_{A} (V)$-valued naive one-form. By an \textit{affine }$A$\textit{-module modeled on the }$A$\textit{-module }$U$ we mean a set $\mathcal{C}$ together with a surjective map $+:\mathcal{C} \times U\rightarrow \mathcal{C} ,\ ( \nabla ,u) \mapsto \nabla +u$ such that $(\nabla+u)+u'=\nabla +(u+u')$ for all $\nabla\in \mathcal{C}$ and $u,u'\in U$. 
With this we can say that the space $\operatorname{Conn}_L  (V)$ of $L$-connections on the $A$-module $V$ is an affine 
$A$-module modeled on the $A$-module $\operatorname{Hom}_{A}( L,A) \otimes _{A}\operatorname{End}_{A} V$. 
A priori it is not clear if $\operatorname{Conn}_L(V) \neq \emptyset $. 
However, if $V$ is projective one can use the construction of Section \ref{subsec:Fedosov} to show that $\operatorname{Conn}_L(V) \neq \emptyset$. Moreover, If $V$ is a module of $A$-endomorphisms of a coherent $A$-module one can use Theorem \ref{thm:canflat} below to show $\operatorname{Conn}_L(V) \neq \emptyset$.
We will discuss other non-projective examples in Sections \ref{sec:flatex} and \ref{sec:LC}.

Let $\alpha ^{\lor } :\Omega _{A|\boldsymbol{k}}\rightarrow \operatorname{Hom}( L,A)$ be the dual to the anchor map $\alpha :L\rightarrow \operatorname{Der}( A)$.
There is an obvious $L$-connection on the $A$-module $A$: 
$$a\mapsto \alpha ^{\lor }(\mathrm{d} a) \otimes 1\in \operatorname{Hom}( L,A) \otimes _{\boldsymbol{k}} A.$$
Two $L$-connections $\nabla ^{U} ,\ \nabla ^{V}$ on the $A$-modules $U$ and $V$, respectively, induce an $L$-connection on $\operatorname{Hom}_{A}( U,V)$ by
 $$\phi \mapsto ( \nabla _{X} \phi )( u) :=\nabla _{X}^{V}( \phi ( u)) -\phi \left( \nabla _{X}^{U} u\right)$$ 
and on $U\otimes V$ by
 $$u\otimes v\mapsto \nabla _{X}( u\otimes v) :=\nabla _{X}^{U} u \otimes v-u\otimes \nabla _{X}^{V} v$$
for $u\in U,\ v\in V,\ \phi \in \operatorname{Hom}_{A}( U,V) ,\ X\in L$. In particular we get $L$-connections on the $A$-modules $V^{\lor } :=\operatorname{Hom}_{A}( V,A)$ and $\operatorname{End}_{A} V$ which are also denoted by $\nabla$.

If $U,V,W$ are $A$-modules with $L$-connections $\nabla^{U} ,\nabla ^{V}$ and $\nabla ^{W}$, respectively, we have
\begin{align}\label{eq:multconn}
\nabla_X ( \phi \psi ) =(\nabla_X \phi ) \psi +\phi ( \nabla_X \psi )
\end{align}
for all $\phi \in \operatorname{Hom}_{A}( U,V) ,\ \psi \in \operatorname{Hom}_{A}( V,W), X\in L$. This is because 
\begin{align*}
( \nabla _{X}( \phi \psi )) u=\nabla _{X}^{W}( \phi \psi u) -\phi \psi \nabla _{X}^{U} u=\nabla _{X}^{W}( \phi \psi u) -\phi \nabla _{X}^{V}( \psi u) +\phi \nabla _{X}^{V}( \psi u) -\phi \psi \nabla _{X}^{U} u
=\nabla _{X}( \phi ) \psi u+\phi ( \nabla _{X} \psi ) u
\end{align*}
holds for all $X\in L, \ u\in U$. A connection $\nabla$ on $\operatorname{End}_A(V)$ satisfying the Leibniz rule Equation \eqref{eq:multconn} for composition will be referred to as \emph{multiplicative}.

\section{Curvature of an $L$-connection}\label{sec:curv}

Let $\nabla $ be an $L$-connection on the $A$-module $V$. By the \textit{curvature} of the $L$-connection $\nabla $
we mean the $\boldsymbol{k}$-trilinear map
$$( X,Y,v) \mapsto ([ \nabla _{X} ,\nabla _{Y}] -\nabla _{[ X,Y]}) v.$$
It is $A$-bilinear and antisymmetric in the first two arguments.

\begin{proposition}
For all $a\in A,\ X,Y\in L$ we have $([ \nabla _{X} ,\nabla _{Y}] -\nabla _{[ X,Y]})( av) =a([ \nabla _{X} ,\nabla _{Y}] -\nabla _{[ X,Y]}) v$, so that
$( X,Y) \mapsto \mathcal{R}( X,Y) :=[ \nabla _{X} ,\nabla _{Y}] -\nabla _{[ X,Y]}$ 
can be interpreted as a naive $\operatorname{End}_{A} (V)$-valued two-form $\mathcal{R} \in \operatorname{Alt}_{A}^{2}\left( L,\operatorname{End}_{A} (V)\right)$, which is referred to as the \emph{curvature endomorphism} of $\nabla $.
\end{proposition}

The $L$-connection $\nabla$ can be extended to an $L$-\textit{covariant derivative} 
$$\nabla :\operatorname{Alt}_{A}^{m}\left( L,V\right)\rightarrow \operatorname{Alt}_{A}^{m+1}\left( L,V\right)$$
using the following version of the Koszul formula
\begin{align} \label{eq:covder}
( \nabla \omega )( X_{0} ,\dotsc ,X_{m})
&=\sum _{i=0}^{m}( -1)^{i}\nabla _{X_{i}}\left( \omega ( X_{0} ,\dotsc ,\widehat{X_{i}} ,\dotsc ,X_{m})\right)\\
\nonumber & \ \ \ \ \ \ \ +\sum _{0\leq i< j\leq m}( -1)^{i+j} \omega \left([ X_{i} ,X_{j}] ,X_{0} ,\dotsc ,\widehat{X_{i}} ,\dotsc ,\widehat{X_{j}} ,\dotsc ,X_{m}\right)
\end{align}
for $X_{0} ,X_{1} ,\dotsc ,X_{m} \in L$.

\begin{theorem}\label{thm:bianchi}
The space $\operatorname{Alt}_{A}\left( L,V\right)$ is a left module over  $\operatorname{Alt}_{A}\left( L,\operatorname{End}_{A} (V)\right)$ with respect to the product $\cup$.
The covariant derivative $\nabla$ is a derivation this left $\operatorname{Alt}_{A}\left( L,\operatorname{End}_{A} (V)\right)$-module structure.
If $\omega \in \operatorname{Alt}_{A}^{m}\left( L,\operatorname{End}_{A} (V)\right)$ then $\nabla ^{2} \omega =\mathcal{R} \cup \omega $. Moreover, the second Bianchi identity $\nabla\mathcal R=0$ holds.
\end{theorem}
\begin{proof}
    This nasty calculation is well-known. Note, however, that $\operatorname{Alt}_{A}^{m}\left( L,V\right)$ with $m=0,1$ do generate $\operatorname{Alt}_{A}\left( L,V\right)$ as a left $\operatorname{Alt}_{A}\left( L,\operatorname{End}_{A} (V)\right)$-module only when $L$ is finitely generated. Also, if $L$ is non-reflexive the $A$-module of naive one-forms is not generated by exact forms.
    This makes the verifications more cumbersome.
\end{proof}

Clearly, if $ \nabla ^{i} \in \operatorname{Conn}_{L}( V_{i})$ are two $ L$-connections, $ i=1,2$, with curvature endomorphisms
$ \mathcal{R}_{\ }^{i} \in \operatorname{Alt}_{A}^{2}\left( L,\operatorname{End}_{A}( V)\right)$ then their sum $ \nabla ^{1} +\nabla ^{2} \in \operatorname{Conn}_{L}( V_{1} \oplus V_{2})$ has curvature endomorphism
 $ \mathcal{R}^{1} +\mathcal{R}^{2}\in\operatorname{Alt}_{A}^{2}\left( L,\operatorname{End}_{A}( V_{1}) \oplus \operatorname{End}_{A}( V_{2})\right)$. The latter $A$-module is seen as a $A$-submodule of $ \operatorname{Alt}_{A}^{2}\left( L,\operatorname{End}_{A}( V_{1} \oplus V_{2})\right)$. Similarly, we get a connection $ \nabla ^{1} \otimes \operatorname{id} +\operatorname{id} \otimes \nabla ^{2} \in \operatorname{Conn}_{L}( V_{1} \otimes V_{2})$ with curvature endomorphism 
$ \mathcal{R}^{1} \otimes \operatorname{id} +\operatorname{id} \otimes \mathcal{R}^{2} \in \operatorname{Alt}_{A}^{2}\left( L,\operatorname{End}_{A}( V_{1} \otimes V_{2})\right)$.

By an \emph{affine connection} we mean a $\operatorname{Der}(A)$-connection. The \emph{torsion} of an affine connection $\nabla$ is defined as the $\boldsymbol{k}$-bilinear antisymmetric map 
\begin{align*}
    (X,Y)\mapsto T_{\nabla}(X,Y):=\nabla_XY-\nabla_YX-[X,Y] 
\end{align*}
It turns out that it is $A$-bilinear, i.e., 
$T_{\nabla}\in\operatorname{Alt}_A^2(\operatorname{Der}(A),\operatorname{Der}(A))$. In fact we have for all $a\in A$
\begin{align*}
    T_{\nabla }( aX,Y) &=\nabla _{aX} Y-\nabla _{Y}( aX) -[aX,Y]\\
&=a\nabla _{X} Y-a\nabla _{Y}( X) -Y( a) X-a[X,Y]+Y( a) X=aT_{\nabla }( X,Y).
\end{align*}
If $\nabla $ is an affine connection then $X\mapsto \nabla _{X} -\frac{1}{2} T_{\nabla }( X,\ )$ defines a torsion-free affine connection. More generally, the torsion tensor $T_{\nabla}\in\operatorname{Alt}_A^2(L,L)$ can be defined for any $\nabla\in\operatorname{Conn}_L(L)$ whenever $(L,A)$ is a Lie-Rinehart algebra.
On an affine variety there is a flat affine connection associated to a system of generators of $\operatorname{Der}(A)$, see Subsection \ref{subsec:adjoint} below. It can be made torsion-free according to the recipe above.

\begin{theorem}\label{thm:canflat}
Let $V$ be a coherent $A$-module (this is, e.g., the case when $A$ is affine). Then $\operatorname{End}_{A}( V)$ admits a canonical flat multiplicative $\operatorname{Der}(A)$-connection $\mathrm{d} :\operatorname{End}_{A}( V)\rightarrow \Omega _{A|\boldsymbol{k}} \otimes _{\boldsymbol{k}}\operatorname{End}_{A}( V)$.
\end{theorem}

\begin{proof}
We interprete a system of generators $v_{1},\dots, v_{\beta_0}$ of $V$ as a \emph{frame} (i.e., a row matrix of elements of $V$) 
$$\boldsymbol{v} =\begin{bmatrix}
v_{1} & \dotsc  & v_{\beta_0}
\end{bmatrix}\in\operatorname{Hom}_A(A^{\beta_0},V)$$
and consider a free resolution
$$F^{0} =A^{\beta _{0}}\xleftarrow{\delta } F^{-1} =A^{\beta _{1}}\xleftarrow{\delta } F^{-2} =A^{\beta _{2}}\leftarrow \dotsc $$
 of $V$ so that the augmented complex 
$$0\leftarrow V\xleftarrow{\boldsymbol{v}} F^{0}\xleftarrow{\delta } F^{-1}\xleftarrow{\delta } F^{-2}\xleftarrow{\delta } \dotsc $$
is exact. A general element $w\in V$ has a non-unique representation $w=\boldsymbol{v}\vec{w}$ with $\vec{w} \in F^{0}$. Since $F^{0}$ is projective we can lift of $\phi \in \operatorname{End}_{A}( V)$ to $F^{0}$, i.e., find a $\phi ^{( 0)} \in \operatorname{End}_{A}\left( F^{0}\right)$ such that $\phi w=\boldsymbol{v} \phi ^{( 0)}\vec{w}$. The difference of two such lifts $\phi ^{( 0)} ,\ \underline{\phi }^{( 0)}$ is an $A$-linear map $F^{0}\rightarrow \delta F^{-1}$, which vanishes after applying $\boldsymbol{v}$. Continuing the process we can lift $\phi $ to the whole $F$ that is there is a $A$-linear map of complexes
$$\Phi =\sum _{i< 0} \phi ^{( i)} :F\rightarrow F,\ \phi ^{( i)} \in \operatorname{End}_{A}( V).$$ This means $\delta \phi ^{( i)} =\phi ^{( i+1)} \delta .$

There is an $A$-linear map 
$\operatorname{Der}( A)\rightarrow \operatorname{Der}\left(\operatorname{End}_{A}\left( A^{m}\right)\right) ,\ X( M)_{i}^{j} :=X\left( M_{i}^{j}\right),$
where we represent an endomorphism $M\in \operatorname{End}_{A}\left( A^{m}\right)$ by the matrix $[ M_{i}^{j}]_{i,j=1,\dotsc ,m}$
such that $M\vec{e}_{i} =\sum M_{i}^{j}\vec{e}_{j}$ for the canonical basis $\vec{e}_{1} ,\dotsc ,\vec{e}_{m}$ of $A^{m}$.
In fact 
\begin{align}\label{eq:matLeib}
X( MN)_{i}^{j} =\sum _{k} X( M_{i}^{k} N_{k}^{i}) =\sum _{k} X( M_{i}^{k}) N_{k}^{i} +\sum _{k} M_{i}^{k} X( N_{k}^{i}) =( X( M) N+MX( N))_{i}^{j}
\end{align}
for $M,N\in \operatorname{End}_{A}\left( A^{m}\right)$. Let us also consider the corresponding universal derivation
$$\mathrm{d} :\operatorname{End}_{A}\left( A^{m}\right)\rightarrow \Omega _{A|\boldsymbol{k}}^{m\times m} \simeq \Omega _{A|\boldsymbol{k}} \otimes _{A}\operatorname{End}_{A}\left( A^{m}\right) ,\ \mathrm{d}( M)_{i}^{j} :=\mathrm{d}\left( M_{i}^{j}\right).$$

Let now $V$ be an $A$-module with free resolution $(F,\delta)$. We claim that  the $A$-linear map $\operatorname{Der}( A)\rightarrow \operatorname{Der}\left(\operatorname{End}_{A}( V)\right) ,\ X( \phi ) w=\boldsymbol{v} X( \phi ^{( 0)})\vec{w}$
for $\phi \in \operatorname{End}_{A}( V) ,\ X\in \operatorname{Der}( A) ,\ w=\boldsymbol{v}\vec{w} \in V$ is well defined.
In order to show this we inspect
$$\boldsymbol{v} X( \phi ^{( 0)})\vec{w} \in \boldsymbol{v} X( \phi ^{( 0)})\vec{w} +\boldsymbol{v} X( \phi ^{( 0)}) \delta F^{-1}.$$
We would like to lift $X( \phi ^{( 0)})$ to $F^{-1}$. But since $F^{-1}$ is projective we can find such a lift $\Xi ^{( -1)} \in \operatorname{End}_{A}\left( F^{-1}\right)$ such that $X( \phi ^{( 0)}) \delta =:-\delta \Xi ^{( -1)}$. 
But then 
$$\boldsymbol{v} X( \phi ^{( 0)}) \delta F^{-1} \in \boldsymbol{v} \delta \Xi ^{-1} F^{-1} =\{0\} ,$$
so that $\boldsymbol{v} X( \phi ^{( 0)})\vec{w}$ is independent of the choice of the lift $\phi ^{( 0)}$. Let us now analyse
$$\phi \psi w=\boldsymbol{v} \phi ^{( 0)}\overrightarrow{\boldsymbol{v} \psi ^{( 0)}\vec{w}} \in \boldsymbol{v} \phi ^{( 0)} \psi ^{( 0)}\vec{w} +\boldsymbol{v} \phi ^{( 0)} \delta F^{-1}.$$
But $\boldsymbol{v} \phi ^{( 0)} \delta F^{-1} =\boldsymbol{v} \delta \phi ^{( -1)} F^{-1} =\{0\}$ and therefore it follows that $\phi \psi w=\boldsymbol{v} \phi ^{( 0)} \psi ^{( 0)}\vec{w}$. From this and \ref{eq:matLeib} we can show the Leibniz rule for the composition $\phi \psi $
$$X( \phi \psi ) =X( \phi ) \psi +\phi X( \psi )$$ for $\phi ,\psi \in \operatorname{End}_{A}( V),$
and the proof of the claim is established.

With this we can define the universal derivation $\mathrm{d} :\operatorname{End}_{A}( V)\rightarrow \Omega _{A|\boldsymbol{k}} \otimes _{A}\operatorname{End}_{A}( V) ,\ $$\langle \mathrm{d} \phi ,X\rangle =X( \phi ) =:\mathrm{d}_{X} \phi $ for $\phi \in \operatorname{End}_{A}( V) ,\ X\in \operatorname{Der}( A)$. 
Since 
$$[\mathrm{d}_{X} ,\mathrm{d}_{Y}]\phi=X( Y( \phi )) -Y( X( \phi )) =\mathrm{d}_{[ X,Y]} \phi $$ for $\phi \in \operatorname{End}_{A}( V),\ X,Y\in\operatorname{Der}(A)$ the universal derivation $\mathrm{d}$ can be interpreted as a flat connection on $\operatorname{End}_{A}( V)$.
\end{proof}

\begin{theorem}
Let $( L,A)$ be a Lie-Rinehart algebra and $S$ a multiplicatively closed subset in $A$. Then the localization $\left( S^{-1} L,S^{-1} A\right)$ of $(L,A)$ at $S$ is a Lie-Rinehart algebra with bracket
$$\left[\frac{X}{s} ,\frac{Y}{t}\right] =\frac{st[ X,Y] -sX( t) Y+tY( s) X}{s^{2} t^{2}}$$
and anchor $\frac{X}{s}\left(\frac{a}{t} \ \right) =\frac{tX( a) -X( t)}{st^{2}}$ for $X,Y\in L,\ a\in A,\ s,t\in S$. Moreover, if $\nabla $ is an $L$-connection on the $A$-module $V$ with curvature endomorphism $\mathcal{R}$ then the formula
$$\nabla_{\frac{X}{s}}\frac{v}{t} :=\frac{t\nabla _{X} v-X( t) v}{st^{2}}$$ 
for $X\in L,\ v\in V,\ s,t\in S$
defines an $\left( S^{-1} L\right)$-connection on $S^{-1} V$ with curvature $\mathcal{R}\left(\frac{X}{r} ,\frac{Y}{s}\right)\frac{v}{t} =\frac{\mathcal{R}( X,Y) v}{rst}$ for $r,s,t\in S$. 
\end{theorem}

\begin{proof}
We would like to show that
$$X\left(\frac{a}{s}\right) =\frac{sX( a) -aX( s)}{s^{2}}$$
does not depend on the choice of the representative of $\frac{a}{s}$. Now $\frac{a'}{s'} \sim \frac{a}{s}$ exactly when
there is a $b\in A$ such that $b(s'a-sa') =0$. Applying $X$ to this be find $X( b)( sa'-s'a) +b( X( s) a'-X( s') a+sX( a') -s'X( a)) =0$. Multiplying this with $b$ we deduce $b^{2}( a'X( s) -aX( s ') +sX( a') -s'X( a)) =0$, which can be rewritten as
\begin{align}\label{eq:bla}
b^{2} sX( a') -b^{2} s'X( a) =-b^{2} a'X( s) +b^{2} aX( s ')
\end{align}
This means
\begin{align*}
&b^{2}\left( s^{2} s'X( a') -s^{2} a'X( s') \ -s^{\prime 2} sX( a) +s^{\prime 2} aX( s)\right)\\
&=ss'\left( b^{2} sX( a') -b^{2} s'X( a)\right) -b^{2} s^{2} a'X( s') \ +b^{2} s^{\prime 2} aX( s)\\
&=ss'b^{2} aX( s ') -ss'b^{2} a'X( s) -b^{2} s^{2} a'X( s') \ +b^{2} s^{\prime 2} aX( s)\\
&\stackrel{{\mbox{\footnotesize Eqn. \eqref{eq:bla}}}}{=}( bs' a-bsa') bsX( s ') +( bs'a-bsa') bs'X( s) =0,
\end{align*}
which shows the claim.

Similarly one checks that 
$$\left[ X,\frac{Y}{s}\right] =\frac{s[ X,Y] -X( s) Y}{s^{2}}$$
does not depend on the choice of the representative $\frac{Y'}{s'} \sim \frac{Y}{s}$ by applying $b[ X,\ ]$ to
$b( sY'-s'Y) =0$. Now it follows that
$$\left[\frac{X}{s} ,\frac{Y}{t}\right] =\frac{st[ X,Y] -sX( t) Y+tY( s) X}{s^{2} t^{2}}$$
is correctly defined. We leave it to the reader to verify the axioms for $\left( S^{-1} L,S^{-1} A\right)$ as well as the claims about the connection. 
\end{proof}

\section{Gauge group action on $\operatorname{Conn}_L(V)$} \label{sec:gauge}

The aim is now to introduce a covariant Maurer-Cartan form $\rho _{g}$ and the action of the gauge group $\operatorname{Aut}_{A}( V)$ on $\operatorname{Conn}_{L}( V)$. The results of this section are classical when $L$ and $V$ are projective and the algebra $A$ is smooth. We do not know of any source that treats the general case.


\begin{lemma}
Let $\nabla$  be an $L$-connection on the $A$-module $V$ with curvature endomorphism $\mathcal R$.
For $g\in \operatorname{Aut}_{A}( V)$ the \emph{Maurer-Cartan form} $\rho _{g} :=g\nabla  ( g^{-1}) \in \operatorname{Hom}_A(L,\operatorname{End}_{A}( V))$ satisfies
\begin{enumerate}
    \item $\nabla \rho _{g} +\frac{1}{2}[ \rho _{g} ,\rho _{g}] =g\mathcal Rg^{-1}-\mathcal R$,
\item $\rho _{hg} =h\rho _{g} h^{-1} +\rho _{h}$ for $g,h\in\operatorname{Aut}_A(V)$,
\item $\rho _{\operatorname{id}} =0$.
\end{enumerate}
\end{lemma}
\begin{proof}
In order to prove (1) we use Equation \eqref{eq:multconn} to deduce from $\operatorname{id}=gg^{-1}$ that $\rho_{g} =g\left(\nabla g^{-1}\right) =-(\nabla g) g^{-1}$.
For $X,Y\in L$ we evaluate
\begin{align*}
\nabla \left( g\nabla \left( g^{-1}\right)\right)( X,Y)&=\nabla _{X}\left( g\nabla _{Y}\left( g^{-1}\right)\right) -\nabla _{Y}\left( g\nabla _{X}\left( g^{-1}\right)\right) -g\nabla _{[ X,Y]}\left( g^{-1}\right)\\
&=\nabla _{X}( g) \nabla _{Y}\left( g^{-1}\right) -\nabla _{X}( g) \nabla _{Y}\left( g^{-1}\right) +g[ \nabla _{X} ,\nabla _{Y}]\left( g^{-1}\right) -g\nabla _{[ X,Y]}\left( g^{-1}\right)\\
&=\nabla _{X}( g) \nabla _{Y}\left( g^{-1}\right) -\nabla _{X}( g) \nabla _{Y}\left( g^{-1}\right) +g\left[\mathcal{R}( X,Y) ,g^{-1}\right]
\end{align*}
and
\begin{align*}
\frac{1}{2}\left[ g\nabla \left( g^{-1}\right) ,g\nabla \left( g^{-1}\right)\right]( X,Y)&=g\nabla _{X}\left( g^{-1}\right) g\nabla _{Y}\left( g^{-1}\right) -g\nabla _{Y}\left( g^{-1}\right) g\nabla _{X}\left( g^{-1}\right)\\
&=-\nabla _{X}( g) \nabla _{Y}\left( g^{-1}\right) +\nabla _{X}( g) \nabla _{Y}\left( g^{-1}\right).
\end{align*}
We conclude that $\nabla \left( g\nabla \left( g^{-1}\right)\right) +\frac{1}{2}\left[ g\nabla \left( g^{-1}\right) ,g\nabla \left( g^{-1}\right)\right]=g\left[\mathcal{R} ,g^{-1}\right] =g\mathcal{R} g^{-1} -\mathcal{R}$ which shows (1).

To verify (2) we evaluate $\rho _{hg} =hg\nabla( hg)^{-1} =-(\nabla( hg)) g^{-1} h^{-1} =-( h\nabla g+(\nabla h) g) g^{-1} h^{-1} =h\rho _{g} h^{-1} +\rho _{h}.$
\end{proof}
\begin{remark}
     We write also $\rho _{g}^{\nabla }$ instead of $\rho _{g}$ to indicate the dependence on $\nabla $. For the Maurer-Cartan form associated to the canonical flat connection $\mathrm{d}$ on $\operatorname{End}_{A}( V)$ (see Theorem \ref{thm:canflat}) we write $\rho _{g}^{0} =g\left(\mathrm{d} g^{-1}\right)$. It satisfies $\mathrm d \rho _{g}^{0}+\frac{1}{2}[\rho _{g}^{0},\rho _{g}^{0}]=0$ as well  as (2), (3).
     However, $\rho _{g}^{0}$ does not serve our purposes and we will work with $\rho _{g}^{\nabla }$.
\end{remark}

\begin{proposition}\label{prop:Reta}
Take an $L$-connection $\nabla $ on the $A$-module $V$ and let $\mathcal{R}$ be its curvature endomorphism.
The curvature endomorphism $\mathcal R^{\eta }$ of the connection $\nabla ^{\eta } :=\nabla +\eta $ \ with $\eta \in \operatorname{Hom}_{A}\left( L,\operatorname{End}_{A}( V)\right)$ is given by
$\mathcal{R}^{\eta } =\mathcal{R} +\nabla \eta +\frac{1}{2}[ \eta ,\eta \mathcal{]}$.
\end{proposition}
\begin{proof}
We calculate
\begin{align*}
\left[ \nabla _{X}^{\eta } ,\nabla _{Y}^{\eta }\right] -\nabla _{[ X,Y]}^{\eta }& =[ \nabla _{X} +\eta ( X) ,\nabla _{Y} +\eta ( Y)] -\nabla _{[ X,Y]} -\eta ([ X,Y])\\&=[ \nabla _{X} ,\nabla _{Y}] -\nabla _{[ X,Y]} +(\nabla  \eta )( X,Y) +[ \eta ( X) ,\eta ( Y)]=\left(\mathcal{R} +\nabla \eta +\frac{1}{2}[ \eta ,\eta \mathcal{]}\right)( X,Y) .
\end{align*}
\end{proof}


Fix a $\nabla \in \operatorname{Conn}_{L}( V)$ with curvature endomorphism $\mathcal{R}$. We define the action of an element $g$ of the gauge group $\operatorname{Aut}_{A}( V)$ on an arbitrary $L$-connection $\nabla ^{\eta } =\nabla +\eta \in \operatorname{Conn}_{L}( V)$ by $g\left( \nabla ^{\eta }\right) =\nabla ^{g.\eta }=\nabla+g.\eta$ with 
$$g.\eta =\rho_{g} +g\eta g^{-1},$$
where $\rho_{g}=g(\nabla  g^{-1})$.

\begin{theorem}
Take an $L$-connection $\nabla $ on an $A$-module $V$ and let $\mathcal{R}$ be its curvature endomorphism. Then for each
$\eta \in \operatorname{Hom}_{A}\left( L,\operatorname{End}_{A}( V)\right)$  and $g\in\operatorname{Aut}_A(V)$ we have $\mathcal{R}^{g.\eta } =g\mathcal{R}^\eta g^{-1}$
for the curvature $\mathcal R^{g.\eta }$ of $\nabla ^{g.\eta }$.
\end{theorem}

\begin{proof} We calculate using Proposition \ref{prop:Reta}
\begin{align*}
\mathcal{R}^{g.\eta }& =\mathcal{R} +\nabla (g.\eta) +\frac{1}{2}[ g.\eta ,g.\eta ]\\
&=\mathcal{R} +\nabla \rho _{g}+\nabla \left( g\eta g^{-1}\right) +\frac{1}{2}[ \rho _{g} ,\rho _{g}] +\left[ \rho _{g} ,g\eta g^{-1}\right] +\frac{1}{2} g[ \eta ,\eta ] g^{-1}\\
&=\mathcal{R} +\underbrace{\nabla \rho _{g} +\frac{1}{2}[ \rho _{g} ,\rho _{g}]}_{=g\mathcal{R} g^{-1} -\mathcal{R}} +( \nabla g) \eta g^{-1} +g\eta  \nabla \left( g^{-1}\right) +\underbrace{\left[ \rho _{g} ,g\eta g^{-1}\right]}_{=-( \nabla g) \eta g^{-1} -g\eta \nabla \left( g^{-1}\right)} +g\left( \nabla \eta +\frac{1}{2}[ \eta ,\eta ]\right) g^{-1}\\
&=g\mathcal{R} g^{-1} +g\left( \nabla \eta +\frac{1}{2}[ \eta ,\eta ]\right) g^{-1}.\end{align*}
\end{proof}

\section{Curvature for finitely generated projective modules}\label{sec:Chern}

The goal in this section is to prove Theorem \ref{thm:chern} providing naive de Rham cohomology classes associated to trace polynomials of the curvature endomorphism in the case of a finitely generated projective $A$-module $V$. For a presentation
of the subject for smooth manifolds the reader may consult \cite{BGV} and for basic material on finitely generated projective modules \cite{Lombardi}.

\begin{lemma}
    Let $ V$ be a finitely generated projective $ A$-module. Pick a \emph{coordinate system} $ (\boldsymbol{v } ,\mathcal{I})$, i.e., a system of generators $ v _{i} \in V,\ v^{i} \in \operatorname{Hom}_{A}( V,A)$ with $ i$ ranging in a finite index set $ \mathcal{I}$ such that $ \sum _{i\in \mathcal{I}} v^{i}( v) v _{i} =v$ for all $ v\in V$. Then $ \operatorname{trace}_{V} :\operatorname{End}_{A}( V)\rightarrow A,\ \varphi \mapsto \sum _{i\in \mathcal{I}} v ^{i}( \varphi ( v_{i}))$ does not depend on the choice of $ \boldsymbol{v }$ and satisfies, $ \operatorname{trace}_{V}( \varphi \psi ) =\operatorname{trace}_{V}( \psi \varphi )$ for all $ \varphi ,\psi \in \operatorname{End}_{A}( V)$.
\end{lemma}

\begin{proof}
    Let $ (\boldsymbol{w } ,\mathcal{J})$ be another coordinate system. First, let us verify that
$ \sum _{i} w ^{j}( v _{i}) v^{i} =w ^{j}$ \ for all $ j\in \mathcal{J}$. But this holds because for all $ v\in V$ we have $ w ^{j}( v) =w ^{j}\left(\sum _{i} v^{i}( v) v _{i}\right) =\left(\sum _{i} w ^{j}( v _{i}) v^{i}\right)( v)$.
We have for all $ j\in \mathcal{J}$ that $ \sum _{i\in \mathcal{I}} w ^{j}( v _{i}) w _{j} =v_{i}$ and from what from what has been said we deduce
$ \sum _{i} v ^{i}( \varphi ( v _{i})) =\sum _{i,j} w ^{j}( v _{i}) v ^{i}( \varphi ( w _{j})) =\sum _{j} w ^{j}( \varphi ( w _{j})) .$
The trace property is easy to verify
$ \operatorname{trace}_{V}( \varphi \psi ) =\sum _{i} v ^{i}( \psi ( \varphi ( v _{i}))) =\sum _{i,i'} v ^{i'}( \varphi ( v _{i})) v ^{i}( \psi ( v_{i'})) =\operatorname{trace}_{V}( \psi \varphi ) .$
\end{proof}

\begin{proposition} \label{prop:trace}
    Let $ V$ be a finitely generated projective $ A$-module with coordinate system $ (\boldsymbol{v } ,\mathcal{I})$, $ \nabla \in \operatorname{Conn}_{L}( V)$ and $ \varphi \in \operatorname{End}_{A}( V) ,\ X\in \operatorname{Der}_{A}( A)$. Then $ X\left(\operatorname{trace}_{V}( \varphi )\right) =\operatorname{trace}_{V}( \nabla _{X} \varphi )$. In other words, $ \nabla \left(\operatorname{trace}_{V}\right) =0$ for the $ A$-module morphism $ \operatorname{trace}_{V} :\operatorname{End}_{A}( V)\rightarrow A$.
\end{proposition}

\begin{proof}
    Recall that under the canonical $ A$-linear isomorphism $ \operatorname{End}_{A}( V)\rightarrow V\otimes _{A} V^{*}$ the identity morphism $ \operatorname{id}$ maps to $ \sum _{i} v _{i} \otimes v ^{i}$. We have that for each $ X\in L$
\begin{align*}
0=\nabla _{X}\left(\operatorname{id}\right) =\nabla _{X}(\sum _{i} v _{i} \otimes v ^{i})=\sum _{i} \nabla _{X} v _{i} \otimes v ^{i} +\sum _{i} v _{i} \otimes \nabla _{X} v ^{i}
=\sum _{i,j}\left( \Gamma _{i}^{j}( X) +\Xi_{i}^{j}( X)\right) v _{j} \otimes v ^{i} ,
\end{align*}
where $ \nabla _{X} v _{i} =\sum _{j} \Gamma _{i}^{j}( X) v _{j}$ and $ \nabla _{X} v ^{j} =\sum _{i} \Xi _{i}^{j}( X) v ^{i}$, with connection $ 1$-forms $ \Gamma _{i}^{j} ,\Xi_{i}^{j} \in \operatorname{Hom}_{A}( L,A)$ for $ i,j\in \mathcal{I}$. This means that (see, e.g., \cite[Section IV.4]{Lombardi})
\begin{align}\label{eq:syzrel}
\Xi ^{j}_{i}( X) +\Gamma ^{j}_ {i}( X) =\sum _{\mu } B_{i}^{\mu } s_{\mu }^{j} +\sum _{\nu } C^{j}_{\nu } t^{\nu }_{i}
\end{align}
for some $ B_{i}^{\mu } ,C^{j}_{\nu } \in A$, where $ s_{\mu }^{i}$ and $ t_{j}^{\nu }$ are the syzygies for $ ( v _{i})_{i\in \mathcal{I}}$ and $ \left( v ^{i}\right)_{i\in \mathcal{I}}$, respectively. In other words, for each $ \mu $ we have $ \sum _{i} s_{\mu }^{i} v _{i} =0$ and for each $ \nu $ we have $ \sum _{j} t_{j}^{\nu } v ^{j} =0$. 

Now
$$ X\left(\operatorname{trace}_{V}( \varphi )\right) =X\left(\sum _{i} v ^{i}( \varphi ( v _{i}))\right) =\sum _{i}\left( \nabla _{X} v ^{i}\right)( \varphi ( v _{i})) +\underbrace{\sum _{i} v ^{i}(( \nabla _{X} \varphi )( v _{i}))}_{=\operatorname{trace}_{V}( \nabla _{X} \varphi )} +\sum _{i} v ^{i}( \varphi ( \nabla _{X} v _{i}))$$
holds and we have to show that $ \sum _{i}\left(\left( \nabla _{X} v ^{i}\right)( \varphi ( v _{i})) +\sum _{i} v ^{i}( \varphi ( \nabla _{X} v _{i}))\right)$ vanishes. But using Equation \eqref{eq:syzrel}
\begin{align*}
&\sum _{i}\left(\left( \nabla _{X} v ^{i}\right)( \varphi ( v _{i})) +\sum _{i} v ^{i}( \varphi ( \nabla _{X} v _{i}))\right) =\sum _{i,j}\left( \Xi_{j}^{i}( X) v ^{j}( \varphi ( v _{i})) +\Gamma _{j}^{i} v ^{j}( \varphi ( v _{i}))\right)\\
&=\sum _{i,j,k}\left( \varphi _{i}^{k} \Xi _{j}^{i}( X) v ^{j}( v _{k}) +\varphi _{i}^{k} \Gamma _{j}^{i} v ^{j}( v _{k})\right) =\sum _{i,j,k} \varphi _{i}^{k}\left( \Xi _{j}^{i}( X) +\Gamma _{j}^{i}( X)\right) v ^{j}( v _{k})\\
&=\sum _{i,j,k} \varphi _{i}^{k}\left( B_{j}^{\mu } s_{\mu }^{i} +C_{j}^{\nu }t^{i}_{\nu }\right) v ^{j}( v _{k}) =\sum _{i,j,k} \varphi _{i}^{k} B_{j}^{\mu } s_{\mu }^{i} v ^{j}( v _{k}) =0
\end{align*}
because $ 0=\sum _{i} \varphi \left( s_{\mu }^{i} v _{i}\right) =\sum _{i,k} s_{\mu }^{i} \varphi _{i}^{k} v _{k}$. Here we use the notation $ \varphi ( v _{i}) =\sum _{j} \varphi _{i}^{k} v _{k}$ with $ \varphi _{i}^{j} \in A$ being the matrix elements of $ \varphi $.
\end{proof}

\begin{proposition}
\label{prop:ups}
    Let $ \nabla $ and $ \nabla '$ be $ L$-connections on the $ A$-modules $ V,V'$ and $ \Upsilon \in \operatorname{Hom}_{A}( V,V')$ be covariantly constant, i.e., $ \nabla '\circ \Upsilon -\Upsilon \circ \nabla =0$. Then $ \Upsilon $ extends to the morphism 
$ \operatorname{id} \otimes \Upsilon :\operatorname{Alt}_{A}( L,A) \otimes _{\boldsymbol{k}} V\rightarrow \operatorname{Alt}_{A}( L,A) \otimes _{\boldsymbol{k}} V',$
which commutes with the covariant derivatives.
\end{proposition}

\begin{proof}
For $ $$ X_{0} ,X_{1} ,\dotsc ,X_{m} \in L$ and $ \omega \in \operatorname{Alt}_{A}( L,A) \otimes _{\boldsymbol{k}} V$ we check
\begin{align*}
&\left( \nabla '\left(\left(\operatorname{id} \otimes \Upsilon \right) \omega \right)\right)( X_{0} ,X_{1} ,\dotsc ,X_{m})\\
&=\sum _{l}( -1)^{l} \nabla '_{X_{l}}( \Upsilon ( \omega ))\left( X_{0} ,\dotsc ,\widehat{X_{l}} ,\dotsc ,X_{m}\right)\\
&\ \ \ \ \ \ \ \ \ \ \ \ \ \ \ \ \ \ \ \ \ \ \ \ \ \ \ \ \ \ \ \ \ \ \ \ \ \ +\sum _{k< l}( -1)^{k+l} \Upsilon ( \omega )\left([ X_{k} ,X_{l}],X_0 ,\dotsc ,\widehat{X_{k}} ,\dotsc ,\widehat{X_{l}} ,\dotsc ,X_{m}\right)\\
&=\sum _{l}( -1)^{l}( \Upsilon ( \nabla _{X_{l}} \omega ))\left( X_{0} ,\dotsc ,\widehat{X_{l}} ,\dotsc ,X_{m}\right)\\
&\ \ \ \ \ \ \ \ \ \ \ \ \ \ \ \ \ \ \ \ \ \ \ \ \ \ \ \ \ \ \ \ \ \ \ \ \ \ +\sum _{k< l}( -1)^{k+l} \Upsilon ( \omega )\left([ X_{k} ,X_{l}] ,X_0,\dotsc ,\widehat{X_{k}} ,\dotsc ,\widehat{X_{l}} ,\dotsc ,X_{m}\right)\\
&=\left(\left(\operatorname{id} \otimes \Upsilon \right) \nabla \omega \right)( X_{0} ,X_{1} ,\dotsc ,X_{m}) .
\end{align*}
\end{proof}

We may substitute the curvature endomorphism $\mathcal{R}$ of an $L$-connection $\nabla $ into a polynomial $f\in \boldsymbol{k}[ x]$. If $L$ is a finitely generated $A$-module then $\left(\operatorname{Alt}_{A}( L,A) ,\ \cup \right)$ is a nilpotent algebra and we could actually work more generally with a formal power series, such as $f( x) =\exp( x) \in \boldsymbol{k} \llbracket x\rrbracket $.

\begin{theorem}
\label{thm:chern}
    Let $\nabla ,\ \nabla ^{\eta } =\nabla +\eta $ be $L$-connections on the $A$-module $V$ with $\eta \in \operatorname{Hom}_{A}\left( L,\operatorname{End}_{A}( V)\right)$ and let $\mathcal{R}, \ \mathcal{R}^\eta$ the corresponding curvature endomorphisms. Then for each $f\in \boldsymbol{k}[ x]$ the expression $\operatorname{trace}_{V}( f(\mathcal{R}))$ defines a $d_{\operatorname{dR}}$-closed element in $\operatorname{Alt}_{A}\left( L,\operatorname{End}_{A}( V)\right)$. Moreover, we have 
$$\operatorname{trace}_{V}\left( f\left(\mathcal{R}^{\eta }\right)\right) -\operatorname{trace}_{V}( f(\mathcal{R})) =d_{\operatorname{dR}}\int _{0}^{1}\left(\operatorname{trace}_{V}\left( \eta\cup f'\left( \mathcal{R}^{s}\right)\right)\right) ds$$
so that the naive de Rham cohomology class $\left[\operatorname{trace}_{V}( f(\mathcal{R}))\right] \in \operatorname{H}_{\operatorname{dR}}( L,A)$ is independent of the choice of $\nabla $.
\end{theorem}

\begin{proof}
First observe that using Propositions \ref{prop:trace} and \ref{prop:ups} and the second Bianchi identity (cf. Theorem \ref{thm:bianchi}) we deduce
$$d_{\operatorname{dR}}\left(\operatorname{trace}_{V}( f(\mathcal{R}))\right) =\operatorname{trace}_{V}\left(\left[\operatorname{\nabla } ,f(\mathcal{R})\right]\right) =\operatorname{trace}_{V}\left(\operatorname{\nabla }\mathcal{R} \cup f'(\mathcal{R})\right) =0.$$
Let us introduce a variable $s$ and consider the $L$-connection $\nabla ^{s} =\nabla +s\eta \in \operatorname{Conn}_{L}( V)[ s]$ with curvature endomorphism $\mathcal{R}^{s} \in \operatorname{Alt}_{A}\left( L,\operatorname{End}_{A}( V)\right)[ s]$. Then for $s=1$ we have $\mathcal{R}^{s} =\mathcal{R}^{\eta }.$  Note that 
$$\frac{d\left(\mathcal{R}^{s}\right)}{ds} =\frac{d\left( \nabla {^{s}}^{2}\right)}{ds} =\left[ \nabla ^{s} ,\frac{d\nabla ^{s}}{ds}\right] =\left[ \nabla ^{s} ,\eta \right].$$
We would like to show
$\frac{d}{ds}\operatorname{trace}_{V}\left( f\left(\mathcal{R}^{s}\right)\right) =d_{\operatorname{dR}}\left(\operatorname{trace}_{V}\left( \eta \cup f'\left(\mathcal{R}^{s}\right)\right)\right)$.
It is sufficient to verify this for the monomial $f( x) =x^{m}$
\begin{align*}
\frac{d\operatorname{trace}_{V}\left(\left(\mathcal{R}^{s}\right)^{\cup m}\right)}{ds} &=\sum _{j=0}^{m}\operatorname{trace}_{V}\left(\left(\mathcal{R}^{s}\right)^{\cup j}\frac{d\left(\mathcal{R}^{s}\right)}{ds}\left(\mathcal{R}^{s}\right)^{\cup m-j-1}\right) =m\operatorname{trace}_{V}\left(\frac{d\left(\mathcal{R}^{s}\right)}{ds}\left(\mathcal{R}^{s}\right)^{\cup m-1}\right)\\
&=m\operatorname{trace}_{V}\left(\left[ \nabla ^{s} ,\eta \right]\left(\mathcal{R}^{s}\right)^{\cup m-1}\right) =m\operatorname{trace}_{V}\left([ \nabla ^{s} ,\eta\cup \left(\mathcal{R}^{s}\right)^{\cup m-1}]\right)\\
&\stackrel{\text{\footnotesize Prop. \ref{prop:trace} and \ref{prop:ups}}}{=} md_{\operatorname{dR}}\left(\operatorname{trace}_{V}\left(\eta\cup(\mathcal{ R}{^{s}})^{\cup m-1}\right)\right)
\end{align*}
Integrating we deduce
$$\operatorname{trace}_{V}\left( f\left(\mathcal{R}^{\eta }\right)\right) -\operatorname{trace}_{V}( f(\mathcal{R})) =d_{\operatorname{dR}}\int _{0}^{1}\left(\operatorname{trace}_{V}\left( \eta\cup f'\left( (\mathcal{R}^{s})\right)\right)\right) ds.$$
\end{proof}

Let $ ( L,A)$ be a Lie-Rinehart algebra and assume that $ L$ is a finitely generated $ A$-module. Suppose that $ V$ is a finitely generated projective $ A$-module. Then one can define for $\kappa\in \boldsymbol{k}^\times$ a \textit{Chern character} of $ V$ as
$$ \operatorname{ch}_\kappa( V) :=\left[\operatorname{trace}_{V}(\exp(\kappa\mathcal{ R}))\right] \in \operatorname{H}_{\operatorname{dR}}( L,A),$$ 
where $ \mathcal{R}$ is the curvature endomorphism of some $ \nabla \in \operatorname{Conn}_{L}( V)$. 
In differential geometry people use the normalization $\kappa=1/(2\pi\sqrt{-1})\in \mathbb{C}^\times $.
The following result is due to Hideki Ozeki \cite{Ozeki}.

\begin{corollary}
Suppose that $ V,V'$ are finitely generated projective $ A$-modules and $L$ is finitely generated. Then $ \operatorname{ch}_\kappa( V\oplus V') =\operatorname{ch}_\kappa( V) + \operatorname{ch}_\kappa( V'), \ \operatorname{ch}_\kappa( V\otimes_A V') =\operatorname{ch}_\kappa( V) \cup \operatorname{ch}_\kappa( V')\in\operatorname{H}_{\operatorname{dR}}(L,A)$.
\end{corollary}

\section{Examples of flat connections} \label{sec:flatex}

\subsection{Adjoint representation}\label{subsec:adjoint}
Let $(L,A)$ be a Lie-Rinehart algebra where $L$ is a finitely generated $A$-module.
We associate to a choice of a system $\boldsymbol{\xi }$ of generators $\xi _{1} ,\dotsc ,\xi _{m}$ of $L$ an $L$-connection
$\nabla {^{\boldsymbol{\xi }}}^{\ }$on $L$ by putting $\nabla _{X}^{\boldsymbol{\xi }} \xi _{i} =[ X,\xi _{i}]$ for $i=1,\dotsc ,m$. The structure functions $c_{ij}^{k}$ of $L$ are at the same time the Christoffel symbols of the connection:
$\nabla _{\xi _{i}}^{\boldsymbol{\xi }} \xi _{j} =[ \xi _{i} ,\xi _{j}] =\sum _{k} c_{ij}^{k} \xi _{k} .$
From Jacobi's identity we have that
$$\left[ \nabla _{X}^{\boldsymbol{\xi }} ,\nabla _{Y}^{\boldsymbol{\xi }}\right] \xi _{i} =[ X,[ Y,\xi _{i}]] -[ Y,[ X,\xi _{i}]] =\nabla _{[ X,Y]}^{\boldsymbol{\xi }} \xi _{i}$$
and conclude that $\nabla ^{\boldsymbol{\xi }}$ is flat. Expressed in terms of the generators the flatness can be rewritten as
\begin{align}\label{eqn:cMC}
([ \nabla^{\boldsymbol{\xi }}_{\xi _{i}} ,\nabla^{\boldsymbol{\xi }}_{\xi _{j}}] -\nabla^{\boldsymbol{\xi }}_{[ \xi _{i} ,\xi _{j}]}) \xi _{m} =\sum_l\left(\xi _{i}\left( c_{jm}^{l}\right) -\xi _{j}\left( c_{im}^{l}\right)\right)\xi _{l} +\sum _{k}\left( c_{ik}^{l} c_{jm}^{k} -c_{jk}^{l} c_{im}^{k}\right) \xi _{l} -\sum _{k} c_{ij}^{k} c_{km}^{l} \xi _{l}=0 .
\end{align}
Note that $\mathfrak{gl}_{m}\left(\operatorname{C}_{\operatorname{dR}}( L,A)\right)$ 
forms a dg Lie $\boldsymbol{k}$-algebra with differential $d_{\operatorname{dR}}$ and super Lie bracket $[\ ,\ ]$ made from $\cup$ and the commutator of matrices.
Let $c\in \mathfrak{gl}_{m}\left(\operatorname{C}_{\operatorname{dR}}^{1}( L,A)\right)$ be the form that sends $\xi _{i}$ to the matrix $c_{i} :=\left[ c_{ij}^{k}\right]_{j,k=1,\dotsc ,m} \in \mathfrak{gl}_{m}( A)$.

Inspecting Equation \eqref{eqn:cMC} we deduce the following.

\begin{theorem}
The flatness of $\nabla ^{\boldsymbol{\xi }}$ is equivalent to the vanishing of $d_{\operatorname{dR}} c-[ c,c] \in \mathfrak{gl}_{m}\left(\operatorname{C}_{\operatorname{dR}}^{2}( L,A)\right)$.
\end{theorem}

In the case when $A=P/I$ is an affine $\boldsymbol{k}$-algebra and $L=\operatorname{Der}(A)$ we would like to discuss the torsion tensor $T_{\nabla ^{\boldsymbol{\xi }}}$ of $\nabla ^{\boldsymbol{\xi }}$. To this end we pick representatives $\sum _{j} \xi _{i}^{j} \partial /\partial x^{j} \in \operatorname{Der}_{I}( P) ,\ \xi _{i}^{j} \in P,$ such that $\xi _{i} =\sum _{j} \xi _{i}^{j} \partial /\partial x^{j} +I\operatorname{Der}( P)$. We leave it to the reader to check that
 $$T_{\nabla ^{\boldsymbol{\xi }}}( X,Y) =\sum _{i,j}\left( Y^{i} X( \xi _{i}^{j}) -X^{i} Y( \xi _{i}^{j})\right)\frac{\partial }{\partial x^{j}}+I\operatorname{Der}( P)$$
where $X=\sum _{i} X^{i} \partial /\partial x^{i} +I\operatorname{Der}( P) ,\ Y=\sum _{i} Y^{i} \partial /\partial x^{i} +I\operatorname{Der}( P) \ \in \operatorname{Der}_{I}( P) /I\operatorname{Der}( P) .$

\subsection{Bott connection}
Here we investigate the Bott connection on the conormal module of an affine algebra $P/I$, where $P=\boldsymbol{k}\left[ x^{1} ,\dotsc ,x^{n}\right]$ and $I=( f_{1} ,\dotsc ,f_{\ell }) \subseteq P$ is an ideal.
The \emph{Bott connection} $\nabla ^{\mathrm{B}}$ on $I/I^{2}$ is given by the formula 
$$\nabla _{X}^{\mathrm{B}}\left( f+I^{2}\right) =\widehat{X}( f) +I^{2}$$ 
for $f\in I$ and a representative $\widehat{X}\in \operatorname{Der}_I( P)$ of $X\in \operatorname{Der}( P/I)\simeq \operatorname{Der}_I( P)/I\operatorname{Der}( P)$. It follows from
$[ \nabla _{X} ,\nabla _{Y}]\left( f+I^{2}\right) =\widehat{X}(\widehat{Y}( f)) -\widehat{Y}( \widehat{X}( f)) +I^{2} =\nabla _{[ X,Y]}\left( f+I^{2}\right)$
for $X,Y\in \operatorname{Der}( P/I) ,\ f\in I$ that $\nabla ^{\mathrm{B}}$ is flat. By Vascolcelos' Theorem (see \cite{Matsumura}) $I/I^2$ is a projective $P/I$-module if and only if $I$ is the ideal of a local complete intersection. Moreover, $I/I^2$ is a free $P/I$-module with basis $f_{1}+I^2,\dotsc ,f_{\ell }+I^2$ if and only if the Koszul complex on $f_{1} ,\dotsc ,f_{\ell }$ is acyclic. 

We would like to analyse what this means in terms of generators of $I$ and $\operatorname{Der}( P/I)$. 
Let $X_{1} ,\dotsc ,X_{m} \in \operatorname{Der}_{I}( P)$ be generators of the $P$-module $\operatorname{Der}_{I}( P)$ of derivations of $P$ tangent to $I$. Since $\operatorname{Der}_{I}( P)$ is closed under the Lie bracket there are structure functions $c_{ij}^{k} \in P$ such that $[ X_{i} ,X_{j}] =\sum _{k} c_{ij}^{k} X_{k}$. We have Jacobi's identity 
$$\sum _{l} c_{ij}^{l} c_{lk}^{m} +\operatorname{cyclic}( i,j,k) =0$$ 
for $i,j,k=1,\dots,\ell$ since $\left(\operatorname{Der}_{I}( P) ,P\right)$
forms a sub Lie-Rinehart algebra of $\left(\operatorname{Der}_{I}( P) ,P\right)$. 
Associated to the choice of generators $f_{1} ,\dotsc ,f_{\ell }$ there are functions $\Gamma _{i\mu }^{\nu } \in P$, $i=1,\dotsc ,\ell ,\ \mu ,\nu =1,\dotsc ,m$, such that
$X_{i}( f_{\mu }) =\sum _{\nu } \Gamma _{i\mu }^{\nu } f_{\nu }$.
The Bott connection can be understood as
$$\nabla _{X_{i}}^{\mathrm{B}}\left( f_{\mu } +I^{2}\right) =\sum _{\nu } \Gamma _{i\mu }^{\nu } f_{\nu } +I^{2} .$$
We refer to the Christoffel symbols $\Gamma _{i\mu }^{\nu } +I\in P/I$ \ of the Bott connection as the \textit{Bottoffel symbols}. From
$$\nabla^{\mathrm{B}}_{X_{i}}\left( \nabla^{\mathrm{B}}_{X_{j}}\left( f_{\mu } +I^{2}\right)\right) =X_{i}\left(\sum _{\nu } \Gamma _{j\mu }^{\nu } f_{\nu }\right) +I^{2} =\sum _{\nu } X_{i}\left( \Gamma _{j\mu }^{\nu }\right) f_{\nu } +\sum _{\nu ,\lambda } \Gamma _{j\mu }^{\nu } \Gamma _{i\nu }^{\lambda } f_{\lambda } +I^{2}$$
we deduce that the zero curvature condition means
\begin{align}
\label{eq:Bottflat}
&([ \nabla^{\mathrm{B}}_{X_{i}} ,\nabla^{\mathrm{B}}_{X_{j}}] -\nabla^{\mathrm{B}}_{[ X_{i} ,X_{j}]})\left( f_{\mu } +I^{2}\right)\\
\nonumber & \ \ \ \ \ \ \ =\sum _{\nu }\left( X_{i}( \Gamma _{j\mu }^{\nu }) -X_{j}( \Gamma _{i\mu }^{\nu })\right) f_{\nu } +\sum _{\nu ,\lambda }\left( \Gamma _{j\mu }^{\nu } \Gamma _{i\nu }^{\lambda } -\Gamma _{i\mu }^{\nu } \Gamma _{j\nu }^{\lambda }\right) f_{\lambda } -\sum _{k,\nu } c_{ij}^{k} \Gamma _{k\mu }^{\nu } f_{\nu } +I^{2} \subseteq I^{2}
\end{align}
for $i,j=1,\dotsc ,\ell ,\ \mu =1,\dotsc ,m$.

We would like to reformulate this in terms of the homological algebra of the de Rham complex
$\operatorname{C_{\operatorname{dR}}}( P/I) =\operatorname{Alt}_{P/I}\left(\operatorname{Der}( P/I) ,P/I\right)$. Note that $\left(\operatorname{C_{\operatorname{dR}}}( P/I) ,d_{dR} ,\cup \right)$ forms a supercommutative dg algebra. Consequently, the space of naive $\mathfrak{gl}_{\ell }$-valued forms $\left(\mathfrak{gl}_{\ell }\left(\operatorname{C_{\operatorname{dR}}}( P/I)\right) ,\mathrm{d} ,\ [ \ ,\ ]\right)$ carries the structure of a dg Lie $\boldsymbol{k}$-algebra. Let $\Gamma \in \mathfrak{gl}_{\ell}\left(\operatorname{C_{\operatorname{dR}}^{1}}( P/I)\right)$ be the naive $\mathfrak{gl}_{\ell}$-valued one-form that sends
$\xi _{i} +I\operatorname{Der}( P)$ to $\left[ \Gamma _{i\mu }^{\nu } +I\right]_{\mu ,\nu =1,\dotsc ,\ell} \in \mathfrak{gl}_\ell( P/I)$. 

\begin{theorem}
The flatness of $\nabla^{\mathrm B}$ is equivalent to the vanishing of
$d_{dR} \Gamma -[ \Gamma ,\Gamma ] \in \mathfrak{gl}_{\ell}\left(\operatorname{C_{\operatorname{dR}}^{2}}( P/I)\right)$.
\end{theorem}

\subsection{Poisson connection on the conormal module of a Poisson ideal}
Here we recall the Poisson connection of \cite{Nambuffel, higherKoszul} on the conormal module of an affine Poisson $\boldsymbol{k}$-algebra. Let $\left( P=\boldsymbol{k}\left[ x^{1} ,\dotsc ,x^{n}\right] ,\ \{\ ,\ \}\right)$ be a Poisson $\boldsymbol{k}$-algebra and $I\subseteq P$ a Poisson ideal, i.e., $\{I,P\} \subseteq I$. Note that $P/I$ becomes a Poisson algebra with bracket $\{f+I,g+I\} =\{f,g\} +I$ for $f,g\in P$. We have a morphism of Lie algebras
$P\rightarrow \operatorname{Der}( P) ,\ f\mapsto X_{f} :=\{f,\ \}$ which descends to a morphism of Lie algebras
$P/I\rightarrow \operatorname{Der}( P/I) ,\ f+I\mapsto X_{f+I} :=\{f+I,\ \}$.
Accordingly, we have a Lie-Rinehart subalgebra $((P/I) X_{P/I} ,P/I)$ of $\left(\operatorname{Der}( P/I) ,P/I\right)$, where $(P/I) X_{P/I}$ denotes the $P/I$-submodule of $\operatorname{Der}(P/I)$ generated by the $X_{f+I}$, $f\in P$.

A $(P/I)X_{P/I}$-connection is referred to as a Poisson connection and the de Rham complex of $((P/I) X_{P/I} ,P/I)$
is called the complex of Poisson cohomology $\left(\operatorname{C}_{\operatorname{Poiss}}( P/I) ,d_{\operatorname{Poiss}}\right) =\operatorname{C}_{\operatorname{dR}}((P/I)X_{P/I} ,P/I) ,\ d_{\operatorname{dR}})$. The restriction of the Bott connection to $X_{P/I}$ gives rise to the Poisson connection
$$\nabla _{X_{f+I}}^{\operatorname{Poiss}}\left( g+I^{2}\right) =\{f,g\} +I^{2} ,\ f\in P,g\in I,$$
where$\ f\in P,g\in I$. This means that $\nabla ^{\operatorname{Poiss}}$ is flat. Let us pick a system of generators $f_{1} ,\dotsc ,f_{\ell }$ of $I$.
The classes $Z_{\mu }^{i\nu } +I\in P/I$ of the coefficients $Z_{\mu }^{i\nu }$ in$\left\{x^{i} ,f_{\mu }\right\} =\sum _{j} Z_{\mu }^{i\nu } f_{\nu }$ can be interpreted as Christoffel symbols of $\nabla ^{\operatorname{Poiss}}$ since $\nabla _{X_{x^{i} +I}}^{\operatorname{Poiss}}\left( f_{\mu } +I^{2}\right) =\left\{x^{i} ,f_{\mu }\right\} +I^{2}$. They are referred to as the \textit{Poissoffel symbols. }

We call $Z$ the $A$-linear map that sends $\mathrm{d} (x^{i}+I) \in \operatorname{Hom}_{P/I}( (P/I)X_{P/I} ,P/I)$ to the $\ell \times \ell $-matrix $Z^{i} :=\left[ Z_{\mu }^{i\nu } +I\right] \in \mathfrak{gl}_{\ell }( P/I)$. Note that $Z$ is an element of the dg Lie algebra $\mathfrak{gl}_{\ell }\left(\operatorname{C}_{\operatorname{Poiss}}( P/I)\right)$
of homological degree $1$. The flatness of $\nabla ^{\operatorname{Poiss}}$ is equivalent to the vanishing of $d_{\operatorname{Poiss}} Z-[ Z,Z] \in \mathfrak{gl}_{\ell }\left(\operatorname{C}_{\operatorname{Poiss}}^{2}( P/I)\right)$.

In \cite{Nambuffel} there has been presented a generalization of this to Nambu-Poisson algebras.

\subsection{The $I/I^2$-connection on 
the Kähler differentials
associated to a first class ideal $I$}

Let $( P,\{\ ,\ \})$ be a Poisson $\boldsymbol{k}$-algebra, where $P=\boldsymbol{k}[x^1,\dots,x^n]$, and put $\Pi^{ij}=\{x^i,x^j\}$. 
By a \textit{first class ideal}\footnote{Differential geometers use the term \emph{coisotropic} ideal and people working in D-modules talk of \emph{involutive} ideals. The idea of first class contraints, however, goes back to P.A.M. Dirac \cite{Dirac}.}
we mean an ideal $I\subseteq P$ such that $\{I,I\} \subseteq I$.
Choosing generators $f_1,\dots,f_\ell$ for $I$ the first class condition boils down to the existence of structure functions $c_{\mu\nu}^\lambda\in P$ such that $\{f_\mu,f_\nu\}=\sum_\lambda c_{\mu\nu}^\lambda f_\lambda$. An important special case of first class ideals is provided by ideals generated by the components of moment maps of Hamiltonian group actions (see, e.g., \cite{HSScompositio}).
One easily verifies that if $I$ is a first class ideal 
$\left( I/I^{2} ,A=P/I\right)$ forms a Lie-Rinehart algebra with bracket $[f+I^2,g+I^2]=\{f,g\}+I^2$ and anchor
$\alpha(f+I^2)(h+I)=\{f,h\}+I$ for $f,g\in I, h\in P$. 

We would like to construct an $I/I^{2}$-connection on $\Omega _{A|\boldsymbol{k}}$ depending on a set of generators $f_{1} ,\dotsc ,f_{\ell } \in P$ of the ideal $I$. In fact, this can be done by putting 
$$\nabla _{f_{\mu } +I^{2}}^{\mathrm{D}}(\mathrm{d}( g+I)):=\mathrm{d}(\{f_{\mu } ,g\} +I)$$
for $g\in P$ and $\mu =1,\dotsc ,\ell $. We refer to $\nabla ^{\mathrm{D}}$ as the \textit{Dirac connection} since the study of first class ideals goes back to the work of P.A.M. Dirac. Using Jacobi's identity for $\{\ ,\ \}$ we see from
\begin{align*}
\left[ \nabla _{f_{\mu } +I^{2}}^{\mathrm{D}} ,\nabla _{f_{\nu } +I^{2}}^{\mathrm{D}}\right](\mathrm{d}( g+I))& =\mathrm{d}(\{f_{\mu } ,\{f_{\nu } ,g\} -\{f_{\nu } ,\{f_{\mu } ,g\} +I)\\
\nabla _{\left[ f_{\mu } +I^{2} ,f_{\nu } +I^{2}\right]}^{\mathrm{D}}(\mathrm{d}( g+I))& =\mathrm{d}(\{\{f_{\mu } ,f_{\nu }\} ,g\} +I)
\end{align*}
that our $I/I^{2}$-connection is flat. From 
\begin{align*}
\nabla _{f_{\mu } +I^{2}}^{\mathrm{D}}\mathrm{d}\left( x^{i} +I\right) &=\sum _{k}\mathrm{d}\left( \Pi ^{ki}\frac{\partial f_{\mu }}{\partial x^{k}} +I\right) =\sum _{j,k}\left(\frac{\partial \Pi ^{ki}}{\partial x^{j}}\frac{\partial f_{\mu }}{\partial x^{k}} +\Pi ^{ji}\frac{\partial ^{2} f_{\mu }}{\partial x^{j} \partial x^{k}} +I\right)\mathrm{d}\left( x^{j} +I\right)\\
\nonumber &=:\sum _{j}\left( \Phi _{\mu j}^{i} +I\right)\mathrm{d}\left( x^{j} +I\right)
\end{align*}
we see that the Christoffel symbols of $\nabla ^{\mathrm{D}}$, which we refer to as \emph{Diraffel symbols}, turn out to be the classes $\Phi_{\mu j}^{i} +I\in P/I$ of 
$$\Phi _{\mu j}^{i} =\sum _{j,k}\left(\frac{\partial \Pi ^{ki}}{\partial x^{j}}\frac{\partial f_{\mu }}{\partial x^{k}} +\Pi ^{ji}\frac{\partial ^{2} f_{\mu }}{\partial x^{j} \partial x^{k}}\right)=\frac{\partial\{f_\mu,x^i\}}{\partial x^j}\in P.$$

The zero curvature condition can be expressed as
\begin{align}\label{eq:DirMC}
&\left(\left[ \nabla _{f_{\mu } +I^{2}}^{\mathrm{D}} ,\nabla _{f_{\nu } +I^{2}}^{\mathrm{D}}\right] -\nabla _{\left[ f_{\mu } +I^{2} ,f_{\mu } +I^{2}\right]}^{\mathrm{D}}\right)\mathrm{d}\left( x^{i} +I\right)\\
\nonumber &=\sum _{k}\left(\left\{f_{\mu } ,\Phi _{\nu k}^{i}\right\} -\left\{f_{\nu } ,\Phi _{\mu k}^{i}\right\} +\sum _{j}\left( \Phi _{\mu k}^{j} \Phi _{\nu j}^{i} -\Phi _{\nu k}^{j} \Phi _{\mu j}^{i}\right) -\sum _{\lambda } c_{\mu \nu }^{\lambda } \Phi _{\lambda k}^{i}\right)\mathrm{d}\left( x^{k} +I\right)=0
\end{align}
for all $\mu ,\nu =1,\dotsc ,\ell ,\ i=1,\dotsc ,n$. Let $\Phi $ be the element of $\mathfrak{gl}_{n}\left(\mathrm{C}_{\operatorname{dR}}^{1}\left( I/I^{2} ,P/I\ \right)\right)$
that sends $f_\mu +I^2$ to the matrix $\left[ \Phi _{\mu j}^{i}\right]_{i,j=1,\dotsc ,n}$.

\begin{theorem}
The flatness of $\nabla^{\mathrm D}$ is equivalent to the vanishing of $d_{\operatorname{dR}} \Phi +[ \Phi ,\Phi ] \in \mathfrak{gl}_{n}\left(\mathrm{C}_{\operatorname{dR}}^{2}\left( I/I^{2} ,\ P/I\right)\right)$.
\end{theorem}

\section{Flat connections and Lie-Rinehart modules} \label{sec:LRmod}

Let $V$ be an $A$-module. The \emph{Atiyah-Lie-Rinehart algebra} $(\mathcal{D}( V) ,A)$ is defined as follows. An element $D\in \operatorname{End}_{\boldsymbol{k}}( V)$ is called a \textit{derivation over }$X\in \operatorname{Der}( A)$ if for all $v\in V,\ a\in A$ the Leibniz rule
$$D( av) =X( a) v+aD( v)$$
holds. Let us put $\mathcal{D}( V) :=\left\{D\in \operatorname{End}_{\boldsymbol{k}}( V) \ \mid \exists X\in \operatorname{Der}( A) :\ D\ \text{is a derivation over } X\right\}$.
The commutator of $D_{1} ,D_{2} \in \mathcal{D}( V)$ is defined to be $[ D_{1} ,D_{2}]( v) :=D_{1}( D_{2}( v)) -D_{2}( D_{1}( v))$.
Obviously $(\mathcal{D}( V), [\ ,\ ])$ forms a Lie algebra. 

If $D\in \operatorname{End}_{\boldsymbol{k}}( V)$ is a derivation over $X,Y\in \operatorname{Der}( A)$ then for all $v\in V,\ a\in A$ the relation
$X( a) v=Y( a) v$ must hold. If $\operatorname{Ann}_{A}( V) =\{0\}$, then it follows that $X=Y$.
For the rest of this section we assume that $\operatorname{Ann}_{A}( V) =\{0\}$ so that there is a \textit{symbol map }$\sigma :\mathcal{D}( V)\rightarrow \operatorname{Der}( A) ,\ D\rightarrow X$, where $X$ is the unique derivation such that $D$ is a derivation over $X$. Note that the symbol map $\sigma$ is the anchor of the Lie-Rinehart algebra $(\mathcal{D}( V),A)$.
It is part of the exact sequence 
$$\operatorname{Der}( A)\xleftarrow{\sigma }\mathcal{D}( V)\leftarrow \operatorname{End}_{A}( V)\leftarrow 0$$
of Lie-Rinehart algebras over $A$.

\begin{definition}
Let $( L,A)$ be a Lie-Rinehart algebra and $V$ an $A$-module such that $\operatorname{Ann}_{A}( V) =\{0\}$.
By an \textit{$( L,A)$-module structure} on $V$ we mean a morphism $\nabla :L\rightarrow \mathcal{D}( V)$ of Lie-Rinehart algebras over $A$.
\end{definition}

\begin{proposition}\label{prop:AtiyaConn}
Let $( L,A)$ be a Lie-Rinehart algebra and $V$ an $A$-module such that $\operatorname{Ann}_{A}( V) =\{0\}$. Then an $(L,A)$-module $V$ is nothing but a flat $L$-connection $\nabla \in \operatorname{Conn}_{L}( V)$. More generally, an $L$-connection is the same thing as an $A$-linear morphism $L\rightarrow \mathcal{D}( V)$ compatible with the anchors.
\end{proposition}

\begin{proof}

Let $\nabla $ be an $L$-connection. The flatness of $\nabla $ means nothing but that the map $\nabla :L\rightarrow \mathcal{D}( V) ,\ X\mapsto \nabla _{X}$ is compatible with the brackets. Conversely, if $\nabla :L\rightarrow \mathcal{D}( V) ,\ X\mapsto \nabla _{X}$ is a map of Lie-Rinehart algebras then we must also have $X=\sigma ( \nabla _{X})$.
\end{proof} 

The definition of an $L$-connection as in the Proposition \ref{prop:AtiyaConn} on a general $A$-module $V$ already appeared in the work of Hideki Ozeki \cite{Ozeki}.

If $V$ is a Lie-Rinehart module over $( L,A)$ encoded by the flat connection $\nabla \in \operatorname{Conn}_{L}( V)$ then the associated covariant derivative $d_{\operatorname{dR}} =\nabla $ (cf. Equation \eqref{eq:covder}) makes $\operatorname{Alt}_{A}( L,V)$ into a left dg-module over the dg algebra $\operatorname{Alt}_{A}\left( L,\operatorname{End}_{A}( V)\right)$. The complex $\left(\operatorname{C}_{\operatorname{dR}}( L,V) :=\operatorname{Alt}_{A}( L,V) ,d_{\operatorname{dR}}\right)$ is referred to as the naive de Rham complex of the Lie-Rinehart module $V$. George S. Rinehart \cite{Rinehart} studied the case when $L$ is projective. It turns out that if $V$ is projective 
then $V$ is Lie-Rinehart module precisely if $V$ is a left module of the universal enveloping algebra
$U( L,A)$ of $( L,A)$ (for details see \cite{Rinehart}). Among other things he showed that the cohomology of $\operatorname{C}_{\operatorname{dR}}( L,V)$ computes $\operatorname{Ext}_{U( L,A)}( V,A)$. We do not know of any results that prove related claims for non-projective $L$. Let us mention two important special cases. Assume that $V$ is projective and $A$ is smooth. Then $V$ is a $(\operatorname{Der}(A),A)$-module if it is the module of sections of a flat vector bundle over the  smooth variety $\operatorname{Spec}(A)$ and $(\operatorname{C}_{\operatorname{dR}}( L,V),d_{\operatorname{dR}})$ is its algebraic de Rham complex. The second example concerns the case when $A=\boldsymbol{k}$. Then  $(\operatorname{C}_{\operatorname{dR}}( L,V),d_{\operatorname{dR}})$ is known as the cochain complex of Lie algebra cohomology of the module $V$ over the $\boldsymbol{k}$-Lie algebra $L$.

\section{Fedosov connection associated to an idempotent}\label{sec:Fed}
\label{subsec:Fedosov}
Every finitely generated projective module $V$ has an associated connection $\nabla^\theta$ that depends on the idempotent $\theta$ used to define $V$. We learned about it from Boris Fedosov's book \cite{Fedosov} where the idea has been used to present the Levi-Civita connection on a closed submanifold of $\mathbb{R}^{d}$ endowed with the embedding metric.
Let $ V=\theta A^m$ be a finitely generated projective $A$-module defined via the idempotent $ \theta=\theta^{2} \in \operatorname{End}_A\left( A^{m}\right)$.
Then the connection is defined by $ \nabla ^{\theta}:=\theta\mathrm{d} \theta:\Omega_{A|\boldsymbol{k}}\otimes_A A^m\to A^m$ so that for $ v\in V$ we have $ \nabla ^{\theta} v=\theta\mathrm{d} v$. 

\begin{proposition}
    The curvature tensor $ \mathcal{R}^{\theta}\in\operatorname{Alt}_A^2(\operatorname{Der}(A),\operatorname{End}_A(V))$ of $ \nabla ^{\theta} :V\rightarrow \operatorname{Hom}_A\left(\operatorname{Der}( A) ,A\right) \otimes _{\boldsymbol{k}} V$ is given by the formula
$$ v\mapsto\mathcal{R}^{\theta}( X,Y) v=\theta[ X( \theta) ,Y( \theta)] v,$$ where $ X,Y\in \operatorname{Der}( A)$ are supposed to act on the matrix $ \theta$ entrywise (cf. Equation \eqref{eq:matLeib}).
\end{proposition}
\begin{proof}
By $\theta\mapsto X( \theta)$ we actually interprete $X\in \operatorname{Der}( A)$ as an element in $\operatorname{Der}\left(\operatorname{End}_A\left( A^{m}\right)\right)$. From $\theta^{2} =\theta$ we derive $X( \theta) -\theta X( \theta) =X( \theta) \theta$. The result follows from combining the terms in
\begin{align*}
\nabla _{X}^{\theta} \nabla _{Y}^{\theta} v&=\theta X( \theta Y( \theta v)) =\theta X( Y( \theta v)) +\theta X( \theta) Y( \theta v)\\
&=\theta X( Y( \theta v)) +\theta X( \theta) Y( \theta) +\theta X( \theta) \theta Y( v)\\
&=\theta X( Y( \theta v)) -\theta X( \theta ) Y( \theta ) v,\\
-\nabla _{Y}^{\theta } \nabla _{X}^{\theta } v&=-\theta Y( X( \theta v)) +\theta Y( \theta ) X( \theta ) v,\\
-\nabla _{[ X,Y]}^{\theta } v&=-\theta [ X,Y]( \theta v) =\theta [ X,Y]( v) -( \theta [ X,Y] \theta )( v).
\end{align*}    
\end{proof}

\section{Levi-Civita connections on the double cone}\label{sec:LC}

In this section we discuss Levi-Civita connections on the double cone (for semi-Riemannian differential geometry the reader may consult \cite{ONeill}). It turns out that with respect to the embedding metric the system of equations for the Christoffel symbols of the Levi-Civita connection can only be solved after localizing along the singularity. This may come as no  surprise to the specialist. On the other hand, when one is working with appropriately chosen metrics arising from symmetry reductions
then it is possible to construct non-flat examples whose Christoffel symbols are regular functions. The observations made here for the double cone can be worked out for semi-Riemannian orbifolds, viewed as algebraic varieties or as differential spaces. The constructions are expected to generalize to semi-Riemannian orbit spaces \cite{Michor}.

Before we enter the discussion let us fix some terminology.
Let $ A$ be a $ \boldsymbol{k}$-algebra. If $ L$ is an $ A$-module we say a symmetric $ A$-bilinear form $ G:L\times L\rightarrow A$ is \textit{non-degenerate} if the \emph{musical $ A$-module morphism}\footnote{In Ricci calculus a contravariant vector $X^{\mu }$ is turned into a covariant vector $X_{\mu }$ by \emph{lowering} the index using the metric tensor $g_{\mu\nu}$, i.e., by the map $\flat: \left( X^{\mu }\right)_{\mu =1,\dotsc ,n} \mapsto \left( X_{\mu } :=\sum _{\nu } g_{\mu \nu } X^{\nu }\right)_{\mu =1,\dotsc ,n}$.}
$$ \flat :L\rightarrow \operatorname{Hom}_{A}( L,A) ,\ X\mapsto X^\flat:=G( X,\ )$$
is injective. We emphasize that in this context non-degeneracy does not imply that  $\flat$ is onto. By a \textit{semi-Riemannian metric} for the Lie-Rinehart algebra $(L,A)$ we mean a non-degenerate symmetric $ \boldsymbol{k}$-bilinear form $ G:L \times L\rightarrow A$. 

Let $( L,A)$ be a Lie-Rinehart algebra and $G:V\times V\rightarrow A$ be a symmetric $A$-bilinear form. Then an $L$-connection $\nabla $ on $V$ is called \textit{metric }if for each $X\in L,\ u,v\in V$ we have
$X( G( u,v)) =G( \nabla _{X} u,v) +G( u,\nabla _{X} v)$.

\begin{theorem}

Let $(L,A)$ be a Lie-Rinehart algebra and $G:L \times L\rightarrow A$ be a symmetric $A$-bilinear form. Let $X_{1} ,\dotsc ,X_{l}$ be a system of generators for the $A$-module $L$.
Assume that the system of \emph{Koszul equations}
\begin{align}\label{eq:Koszul}
2G( \Xi ,X_{k}) &=X_{i}( G( X_{j} ,X_{k})) +X_{j}( G( X_{k} ,X_{i})) -X_{k}( G( X_{i} ,X_{j})) \ \\
\nonumber &\ \ \ \ \ \ \ \ \ \ \ \ \ \ \ \ \ \ \ \ \ \ \ \ \ \ \ \ \ \ \ \ \ \ \ +G([ X_{i} ,X_{j}] ,X_{k}) -G([ X_{j} ,X_{k}] ,X_{i}) +G([ X_{k} ,X_{i}] ,X_{j})
\end{align}
has a solution $\Xi=\sum_m\Gamma _{ij}^mX_m \in \operatorname{Der}( A)$ for each $i,j,k=1,\dotsc ,l$ with $\Gamma _{ij}^m$. Then $\nabla _{X_{i}} X_{j} :=\sum_m\Gamma _{ij}^mX_m$ defines a metric torsion-free connection $\nabla\in\operatorname{Conn}_L(L)$. If $G$ is non-degenerate the unique $\nabla$ is referred to as the \emph{Levi-Civita connection} of $G$. The coefficients $\Gamma _{ij}^m\in A$ are called the \emph{Christoffel symbols} of $\nabla$. 
\end{theorem}

\begin{proof}
Let us write 
\begin{align*}
F( X,Y,Z) :=X( G( Y,Z)) +Y( G( Z,X)) -Z( G( X,Y)) +G([ X,Y] ,Z) -G([ Y,Z] ,X) +G([ Z,X] ,Y)
\end{align*}
for $X,Y,Z\operatorname{Der}(A)$.
We would like to verify that the map $Z\mapsto F( X_{i} ,X_{j} ,Z)$ is in $\operatorname{Hom}_{A}\left(\operatorname{Der}( A) ,A\right)$. But the $A$-linearity follows from
\begin{align*}
&X_{i}( G( X_{j} ,aZ)) +X_{j}( G( aZ,X_{i})) -G([ X_{j} ,aZ] ,X_{i}) +G([ aZ,X_{i}] ,X_{j})\\
&=aX_{i}( G( X_{j} ,Z)) +X_{i}( a) G( X_{j} ,Z) +aX_{j}( G( Z,X_{i})) +X_{j}( a) G( Z,X_{i})\\
&\ \ \ \ \ \ \ \ \ \ \ \ \ \ -aG([ X_{j} ,Z] ,X_{i}) -X_{j}( a) G( Z,X_{i}) +aG([ Z,X_{i}] ,X_{j}) -X_{i}( a) G( Z,X_{j})\\
&=aX_{i}( G( X_{j} ,Z)) +aX_{j}( G( Z,X_{i})) -aG([ X_{j} ,Z] ,X_{i}) +G([ Z,X_{i}] ,X_{j})
\end{align*}
for $a\in A$. So it is enough to assume the Koszul formula holds for generators. Similarly, we can show that $ $$F( aX_{i} ,X_{j} ,X_{k}) =aF( X_{i} ,X_{j} ,X_{k})$ so that $\nabla _{aX_{i}} X_{j} =a\nabla _{X_{i}} X_{j}$.
To show Leibniz rule $\nabla _{X_{i}}( aX_{j}) =a\nabla _{X_{i}} X_{j} +X_{i}( a) X_{j}$ we calculate
\begin{align*}
&F( X_{i} ,aX_{j} ,X_{k})=X_{i}( G( aX_{j} ,X_{k})) +aX_{j}( G( X_{k} ,X_{i})) -X_{k}( G( X_{i} ,aX_{j})) \\
&\ \ \ \ \ \ \ \ \ \ \ \ \ \ \ \ \ \ \ \ \ \ \ \ \ \ \ \ \ \ \ \ \ \ \ +G([ X_{i} ,aX_{j}] ,X_{k}) -G([ aX_{j} ,X_{k}] ,X_{i}) +G([ X_{k} ,X_{i}] ,aX_{j})\\
&=aF( X_{i} ,X_{j} ,X_{k}) +X_{i}( a) G( X_{j} ,X_{k}) -X_{k}( a) G( X_{i} ,X_{j}) +X_{i}( a) G( X_{j} ,X_{k}) +X_{k}( a) G( X_{j} ,X_{i}) ,
\end{align*}
so that $2G( \nabla _{X_{i}} aX_{j} ,X_{k}) -2aG( \nabla _{X_{i}} X_{j} ,X_{k}) =F( X_{i} ,aX_{j} ,X_{k}) -aF( X_{i} ,X_{j} ,X_{k}) =2X_{i}( a) G( X_{j} ,X_{k})$.

The torsion-freeness of $\nabla $ follows from $F( X_{i} ,X_{j} ,X_{k}) -F( X_{j} ,X_{i} ,X_{k}) =2G([ X_{i} ,X_{j}] ,X_{k})$.

To verify that the connection is metric we check
\begin{align*}
&2G( \nabla _{X_{i}} X_{j} ,X_{k}) +2G( X_{j} ,\nabla _{X_{i}} X_{k}) =F( X_{i} ,X_{j} ,X_{k}) +F( X_{i} ,X_{k} ,X_{j})\\
&=X_{i}( G( X_{j} ,X_{k})) +X_{j}( G( X_{k} ,X_{i})) -X_{k}( G( X_{i} ,X_{j})) \ \\
&\ \ \ \ \ \ \ \ \ \ \ \ \ \ \ \ \ \ \ \ \ \ \ \ \ \ \ \ \ \ \ \ \ \ \ \ \ \ \ +G([ X_{i} ,X_{j}] ,X_{k}) -G([ X_{j} ,X_{k}] ,X_{i}) +G([ X_{k} ,X_{i}] ,X_{j})\\
&\ \ \ \ \ +X_{i}( G( X_{j} ,X_{k})) +X_{k}( G( X_{j} ,X_{i})) -X_{j}( G( X_{i} ,X_{k})) \ \\
&\ \ \ \ \ \ \ \ \ \ \ \ \ \ \ \ \ \ \ \ \ \ \ \ \ \ \ \ \ \ \ \ \ \ \ \ \ \ \ +G([ X_{i} ,X_{k}] ,X_{j}) -G([ X_{k} ,X_{j}] ,X_{i}) +G([ X_{j} ,X_{i}] ,X_{k})=2X_{i}( G( X_{j} ,X_{k})) .
\end{align*}
\end{proof}

The Christoffel symbols are of course not unique since there can be syzygies among the generators $X_{1} ,\dotsc ,X_{l}$.  We observe that the torsion-freeness of $ \nabla $
$$ \nabla _{X_{i}} X_{j} -\nabla _{X_{j}} X_{i} -[ X_{i} ,X_{j}] =\sum _{k}\left( \Gamma _{ij}^{k} -\Gamma _{ji}^{k} -c_{ij}^{k}\right) X_{k}$$
means that $ \Gamma _{ij}^{k} -\Gamma _{ji}^{k} -c_{ij}^{k} =\sum _{m} \lambda _{ijm}^{k} s^{m}$ for some coefficients $ \lambda _{ijl}^{k} \in A$, where $ s^{1} ,\dotsc ,s^{l}$ are the syzygies for $ X_{1} ,\dotsc ,X_{l}$, i.e., $ \sum _{k} s^{k} X_{k} =0$.
Here we write $[X_i,X_j]=\sum_k c_{ij}^kX_k$ with structure functions $c_{ij}^k\in A$.

If $A=P/I$ is an affine algebra one is typically working with representatives $\tilde{\Gamma }_{ij}^{k} \in P$ of the Christoffel symbols $\Gamma _{ij}^{k} =\tilde{\Gamma }_{ij}^{k} +I\in A+P/I$. For $ \nabla $ to be well-defined it is necessary and sufficent that
for all $i,j,\mu,\nu$
\begin{align*}
\sum _{i} s_{\mu }^{i}\tilde{\Gamma }_{ij}^{k} &\in \sum _{\nu } \lambda _{j\mu }^{\nu } s_{\nu }^{k} +I,\\
\sum _{j} s_{\mu }^{j}\tilde{\Gamma }_{ij}^{k} +X_{i}\left( s_{\mu }^{k}\right)&\in \sum _{\nu } \upsilon _{i\mu }^{\nu } s_{\nu }^{k} +I
\end{align*}
for some $ \lambda _{j\mu }^{\nu } ,\ \upsilon _{i\mu }^{\nu } \in P$. Here $ s_{\mu } =\left[ s_{\mu }^{1} \ \dotsc \ s_{\mu }^{l}\right]^{\top }\in P$ denotes a representative of the $ \mu $th syzygy of $ X_{1} ,\dotsc ,X_{l}$. This type of relations appear to be relevant if one tries to lift $L$-connections to obtain homotopy connections on resolutions (see \cite{Vitagliano}).

It is often convenient to note that the Koszul equations \eqref{eq:Koszul} simplify whenever indices coincide.
\begin{align}\label{eq:Koszulsimple}
2G(\nabla_{X_j}X_j,X_i) &=2 X_{j}( G( X_{j} , X_{i})) - X_{i}( G( X_{j} , X_{j})) -2G([ X_{j} , X_{i}] , X_{j}) ,\\
\nonumber 2G(\nabla_{X_j}X_i,X_j) &= X_{i}( G( X_{j} , X_{j})) +2G([ X_{j} , X_{i}] , X_{j}) ,\\
\nonumber 2G(\nabla_{X_i}X_j,X_j) &= X_{i}( G( X_{j} , X_{j})) .
\end{align}

The algebra of regular functions on the double cone $A_{2}$ is the affine algebra $A=P/I$, where $P=\boldsymbol{k}\left[ u^{1} ,u^{2} ,u^{3}\right]$ and $I=\left( u^{1} u^{2} -u^{3} u^{3}\right)$. From Sections 4 and 5 of \cite{HOS} it follows that the four vector fields
\begin{align}\label{eq:derconegens}
X_{1} & :=2u^{1}\cfrac{\partial }{\partial u^{1}} +u^{3}\cfrac{\partial }{\partial u^{3}} ,\ X_{2} :=\ 2u^{3}\cfrac{\partial }{\partial u^{1}} +u^{2}\cfrac{\partial }{\partial u^{3}} ,\ \\
X_{3} & :=2u^{3}\cfrac{\partial }{\partial u^{2}} +u^{1}\cfrac{\partial }{\partial u^{3}} ,\ X_{4} :=2u^{2}\cfrac{\partial }{\partial u^{2}} +u^{3}\cfrac{\partial }{\partial u^{3}}
\end{align}
generate the $P$-module $\operatorname{Der}_{I}( P)=\{X\in\operatorname{Der}(P)\mid X(I)\subseteq I\}$ so that their classes form generators of $\operatorname{Der}( A) =\operatorname{Der}_{I}( P) \ /I\operatorname{Der}( P)$. The syzygies over $A$ among them are given by 
$$\begin{bmatrix}
X_{1} & X_{2} & X_{3} & X_{4}
\end{bmatrix}\begin{bmatrix}
u^{2} & -u^{3} & 0 & 0\\
-u^{3} & u^{1} & 0 & 0\\
0 & 0 & u^{2} & -u^{3}\\
0 & 0 & -u^{3} & u^{1}
\end{bmatrix} =0.$$
If we view $A_{2}$ as a subvariety of flat euclidean space $\operatorname{Spec}( P)$ supplied with the embedding metric then we can record the mutual euclidean scalar product among the $X_{i} ,\ i=1,2,3,4,$ into the table
$$\begin{array}{ c|c c c c }
G_{\operatorname{emb}} & X_{1} & X_{2} & X_{3} & X_{4}\\
\hline
X_{1} & 4u^{1} u^{1} +u^{3} u^{3} & 4u^{1} u^{3} +u^{2} u^{3} & u^{1} u^{3} & u^{3} u^{3}\\
X_{2} & 4u^{1} u^{3} +u^{2} u^{3} & 4u^{3} u^{3} +u^{2} u^{2} & u^{1} u^{2} & u^{2} u^{3}\\
X_{3} & u^{1} u^{3} & u^{1} u^{2} & 4u^{3} u^{3} +u^{1} u^{1} & 4u^{2} u^{3} +u^{1} u^{3}\\
X_{4} & u^{3} u^{3} & u^{2} u^{3} & 4u^{2} u^{3} +u^{1} u^{3} & 4u^{2} u^{2} +u^{3} u^{3}
\end{array}$$
whose entries are the coefficients of the embedding metric $G_{\operatorname{emb}} :\operatorname{Der}( A) \times \operatorname{Der}( A)\rightarrow A$. It turns out that Equation \eqref{eq:Koszul} is only solvable after localizing at the ideal $I$. That is, there is a Levi-Civita connection $\nabla^{\operatorname{emb}} \in \operatorname{Conn}_{\operatorname{Der}(A_{I})}\left(\operatorname{Der}( A_{I})\right)$ but none in  $ \operatorname{Conn}_{\operatorname{Der}(A)}\left(\operatorname{Der}( A)\right)$ . 


On the other hand, when $\boldsymbol{k}=\mathbb{C}$ or $\mathbb{R}$ the double cone can be also understood as a categorical quotient of a finite group. The idea we pursue here is to take an invariant semi-Riemannian metric before dividing out the group action. This in turn gives a semi-Riemannian metric on the categorical quotient and we will show that if the invariant metric is carefully chosen the Koszul equations \eqref{eq:Koszul} can be solved over the algebra of regular functions. The group in question is $\mathbb{Z}_2=\{1,-1\}$ acting on $\boldsymbol{k}^2=T^*\boldsymbol{k}=:W$ by $(q,p)\mapsto(-q,-p)$. The fundamental polynomial invariants of the $\mathbb{Z}_2$-action are given by $u^1=qq,\ u^2=pp$ and $u^3=qp$ satisfying the relation $u^1u^2-u^3u^3=0$. Accordingly, we have a surjective map of $\boldsymbol{k}$-algebras $\boldsymbol{u}^*:\boldsymbol{k}[x^1,x^2,x^3]\to \boldsymbol{k}[W]^{\mathbb{Z}_2},\ x^i\mapsto u^i$ for $i=1,2,3$, so that the invariant ring  $\boldsymbol{k}[W]^{\mathbb{Z}_2}$ is isomorphic to $P/I$, where $\boldsymbol{k}[x^1,x^2,x^3]$ and $I=(x^1x^2-x^3 x^3)$. 
Abusing notation we use the letter $u^i$ for the coordinates $x^i$ as well.
The general $ \mathbb{Z}_{2}$-invariant metric $ g$ on $ W$ is given by
$ \alpha ( q,p)(\mathrm{d} q)^{2} +\beta ( q,p)(\mathrm{d} p)^{2} +2\gamma ( q,p)\mathrm{d} q\mathrm{d} p,$
where $ \alpha ,\beta ,\gamma \in \boldsymbol{k}[ W]^{\mathbb{Z}_{2}} =:R$. This can be also written in matrix form as
$$ \begin{array}{ c|c c }
g & \frac{\partial }{\partial q} & \frac{\partial }{\partial p}\\
\hline
\frac{\partial }{\partial q} & \alpha  & \gamma \\
\frac{\partial }{\partial p} & \gamma  & \beta 
\end{array} .$$
The restriction of $ g$ to the generators of the module of $ \mathbb{Z}_{2}$-invariant vector fields $ q\frac{\partial }{\partial q} ,p\frac{\partial }{\partial q} ,q\frac{\partial }{\partial p} ,p\frac{\partial }{\partial p}$ is given by
$$ \begin{array}{ c|c c c c }
g_{|\operatorname{Der}( R)^{\mathbb{Z}_{2}} \times \operatorname{Der}( R)^{\mathbb{Z}_{2}}} & q\frac{\partial }{\partial q} & p\frac{\partial }{\partial q} & q\frac{\partial }{\partial p} & p\frac{\partial }{\partial p}\\
\hline
q\frac{\partial }{\partial q} & u^{1} \alpha  & u^{3} \alpha  & u^{1} \gamma  & u^{3} \gamma \\
p\frac{\partial }{\partial q} & u^{3} \alpha  & u^{2} \alpha  & u^{3} \gamma  & u^{2} \gamma \\
q\frac{\partial }{\partial p} & u^{1} \gamma  & u^{3} \gamma  & u^{1} \beta  & u^{3} \beta \\
p\frac{\partial }{\partial p} & u^{3} \gamma  & u^{2} \gamma  & u^{3} \beta  & u^{2} \beta 
\end{array} .$$
The matrix of $ g_{|\operatorname{Der}( R)^{\mathbb{Z}_{2}} \times \operatorname{Der}( R)^{\mathbb{Z}_{2}}}$ can be understood as a Kronecker product,
so that $ \det g_{|\operatorname{Der}( R)^{\mathbb{Z}_{2}} \times \operatorname{Der}( R)^{\mathbb{Z}_{2}}} =( u^{1} u^{2} -u^{3} u^{3})^2( \alpha \beta -\gamma \gamma )^2$.

It is easy to check (see \cite[Section 5]{HOS}) that the generators $q\frac{\partial }{\partial q} ,p\frac{\partial }{\partial q} ,q\frac{\partial }{\partial p} ,p\frac{\partial }{\partial p}$ of the $ \boldsymbol{k}[q,p]^{\mathbb{Z}_{2}}$-module $ \operatorname{Der}(\boldsymbol{k}[ p,q])^{\mathbb{Z}_{2}}$ are mapped by $\boldsymbol{u}^{*}$ to $X_{1} ,X_{2} ,X_{3} ,X_{4} \in \operatorname{Der}( A)$. The induced metric on $\operatorname{Der} (A)$ is
\begin{equation*}
\begin{array}{ c|c c c c }
G & X_{1} & X_{2} & X_{3} & X_{4}\\
\hline
X_{1} & u^{1} a & u^{3} a & u^{1} c & u^{3} c\\
X_{2} & u^{3} a & u^{2} a & u^{3} c & u^{2} c\\
X_{3} & u^{1} c & u^{3} c & u^{1} b & u^{3} b\\
X_{4} & u^{3} c & u^{2} c & u^{3} b & u^{2} b
\end{array} ,
\end{equation*}
where $\alpha (q,p)=a(u^{1} ,u^{2} ,u^{3} ),\ \beta (q,p)=b(u^{1} ,u^{2} ,u^{3} )$ and $\gamma (q,p)=c(u^{1} ,u^{2} ,u^{3} )$ with $a,b,c\in P$. Its inverse in the algebra of rational functions $\boldsymbol{k} (a,b,c,u^{1} ,u^{2} ,u^{3} )$ is
\begin{equation*}
G^{-1} =\frac{1}{( ab-cc)\left( u^{1} u^{2} -u^{3} u^{3}\right)}\begin{bmatrix}
u^{2} b & -u^{3} b & -u^{2} c & u^{3} c\\
-u^{3} b & u^{1} b & u^{3} c & -u^{1} c\\
-u^{2} c & u^{3} c & u^{2} a & -u^{3} a\\
u^{3} c & -u^{1} c & -u^{3} a & u^{1} a
\end{bmatrix} .
\end{equation*}
It turns out that if we assume the determinant of $g$ to be a unit, i.e., $ab-cc\in \boldsymbol{k}^{\times }$, then the Koszul equations \eqref{eq:Koszul} are solvable in $A$. We cannot solve, however, the system naively using the above formula for the inverse. This is because the syzygies introduce ambiguities that make linear algebra calculations not viable\footnote{One runs into inconsistencies like
$ X_{3} =\frac{\left( u^{1} u^{2} -u^{3} u^{3}\right) X_{3}}{\left( u^{1} u^{2} -u^{3} u^{3}\right)} =\frac{u^{1} u^{2} X_{3} +u^{3} u^{1} X_{4}}{\left( u^{1} u^{2} -u^{3} u^{3}\right)} =\frac{u^{1} u^{2} X_{3} -u^{1} u^{2} X_{3}}{\left( u^{1} u^{2} -u^{3} u^{3}\right)} =0.$}.
But we mention that in those inconsistent formal manipulations the factor $ u^{1} u^{2} -u^{3} u^{3}$ in the denominator cancels out when applying $ G^{-1}$ to the right hand side of Equation \eqref{eq:Koszul}. In the absence of syzygies this approach is actually mathematically sound (see Section \ref{sec:Dubrovin}). Assuming $ab-cc\in \boldsymbol{k}^{\times }$, it can be easily checked using elementary row operations that the kernel of $G$ coincides with the $A$-module of first syzygies. This means that $G$ is non-degenerate.

Fortunately, assuming $ \boldsymbol{k} =\mathbb{R}$ or $ \mathbb{C}$, we can deduce the Christoffel symbols in a clean and simple way using results from \cite{liftingHomo, HOS}. The approach works for all categorical quotients $ V/\!\!/\Gamma $ of a finite group $ \Gamma $ acting linearly on a vector space $ V$. In hindsight it can be checked that the formulas are correct, independently of the choice of the field $ \boldsymbol{k}$. The result \cite{HOS} we are using is that the canonical map
\begin{align} \label{eq:lambda}
 \lambda :\operatorname{Der}( R)^{\Gamma }\rightarrow \operatorname{Der}\left( R^{\Gamma }\right) ,\ (\lambda ( X) )(f):=X( f) ,
 \end{align}
where $ R=\boldsymbol{k}[ V]$ and $ f\in R^{\Gamma } ,\ X\in \operatorname{Der}( R)^{\Gamma }$, is an isomorphism of $ R^{\Gamma }$-modules. Via the Hilbert map $ \boldsymbol{u} =\left( u^{1} ,\dotsc ,u^{n}\right) :V\rightarrow \boldsymbol{k}^{n}$ we construct an isomorphism $ \boldsymbol{u}^{*} :R^{\Gamma }\rightarrow A:=P/I$ where $ P=\boldsymbol{k}\left[ u^{1} ,\dotsc ,u^{n}\right]$ and $ I$ is an ideal in $ P$ coming from the relations between fundamental polynomial invariants. Accordingly, using the composition of $ \boldsymbol{u}^{*}$ with $ \lambda $ we obtain an isomorphism $ \Phi $ of the Lie-Rinehart algebras $ \left(\operatorname{Der}( R)^{\Gamma } ,R^{\Gamma }\right)$ and $ \left(\operatorname{Der}( A) ,A\right)$. Assuming that $ g$ is a $ \Gamma $-invariant metric on $ \boldsymbol{k}[ V] \simeq \boldsymbol{k}\left[ x^{1} ,\dotsc ,x^{m}\right]$ with Levi-Civita connection $ \nabla ^{g}$ we know that $ \nabla _{\xi }^{g} \eta \in \operatorname{Der}( R)^{\Gamma }$ for all $ \xi ,\eta \in \operatorname{Der}( R)^{\Gamma }$. Consequently, we can for $ X,Y\in \operatorname{Der}( A)$ define 
$$ \nabla _{X} Y:=\Phi \left( \nabla _{\Phi ^{-1}( X)}^{g} \Phi ^{-1}( Y)\right)$$
and note that it satisfies the Koszul equation \eqref{eq:Koszul} for the metric $ G$ given by $ G( X,Y) =g\left( \Phi ^{-1}( X) ,\Phi ^{-1}( Y)\right)$. Obviously, the Riemannian curvature $ \mathcal{R}$ of $ G$ can be expressed
in terms of the Riemannian curvature $ \mathcal{R}^{g}$ of $ g$:
$$ \mathcal{R}( X,Y) Z=\Phi \left(\mathcal{R}^g\left( \Phi ^{-1}( X) ,\Phi ^{-1}( Y)\right) \Phi ^{-1}( Z)\right) ,$$
where $ X,Y,Z\in \operatorname{Der}( A)$. Let $ F_{\mu \nu }^{\lambda }$, $ \mu ,\nu ,\lambda \in \{1,\dotsc ,m\}$, be the Christoffel symbols of $ \nabla ^{g}$, i.e., 
$$ \nabla^g_{\frac{\partial }{\partial x^{\mu }}}\frac{\partial }{\partial x^{\nu }} =\sum _{\lambda } F_{\mu \nu }^{\lambda }\frac{\partial }{\partial x^{\lambda }}$$
and $ F_{\mu } =\left[ F_{\mu\nu }^{\lambda }\right]_{\nu ,\lambda \in \{1,\dotsc ,m\}}$ , $ \mu \in \{1,\dotsc ,m\}$, the corresponding matrices. In the case of the double cone we have $ m=2,$ $ x^{1} =q$ and $ x^{2} =p$. The Christoffel symbols can easily deduced from Equation \eqref{eq:Koszulsimple}. The matrices are
\begin{align*}
F_{1}&=\frac{1}{2(\alpha \beta -\gamma \gamma)}\begin{bmatrix}
\beta  & -\gamma \\
-\gamma  & \alpha 
\end{bmatrix}\begin{bmatrix}
\frac{\partial \alpha }{\partial q} & \frac{\partial \alpha }{\partial p}\\
2\frac{\partial \gamma }{\partial q} -\frac{\partial \alpha }{\partial p} & \frac{\partial \beta }{\partial q}
\end{bmatrix} =\frac{1}{2( \alpha \beta -\gamma \gamma )}\begin{bmatrix}
\beta \frac{\partial \alpha }{\partial q} -2\gamma \frac{\partial \gamma }{\partial q} +\gamma \frac{\partial \alpha }{\partial p} & \beta \frac{\partial \alpha }{\partial p} -\gamma \frac{\partial \beta }{\partial q}\\
-\gamma \frac{\partial \alpha }{\partial q} +2\alpha \frac{\partial \gamma }{\partial q} -\alpha \frac{\partial \alpha }{\partial p} & -\gamma \frac{\partial \alpha }{\partial p} +\alpha \frac{\partial \beta }{\partial q}
\end{bmatrix} ,\\
\\
F_{2} &=\frac{1}{2(\alpha \beta -\gamma \gamma)}\begin{bmatrix}
\beta  & -\gamma \\
-\gamma  & \alpha 
\end{bmatrix}\begin{bmatrix}
\frac{\partial \alpha }{\partial p} &2\frac{\partial \gamma }{\partial p} -\frac{\partial \beta }{\partial q}\\
\frac{\partial \beta }{\partial q} & \frac{\partial \beta }{\partial p}
\end{bmatrix} =\frac{1}{2( \alpha \beta -\gamma \gamma )}\begin{bmatrix}
\beta \frac{\partial \alpha }{\partial p} -\gamma \frac{\partial \beta }{\partial q} & 2\beta \frac{\partial \gamma }{\partial p} -\beta \frac{\partial \beta }{\partial q} -\gamma \frac{\partial \beta }{\partial p}\\
-\gamma \frac{\partial \alpha }{\partial p} +\alpha \frac{\partial \beta }{\partial q} & -2\gamma \frac{\partial \gamma }{\partial p} +\gamma \frac{\partial \beta }{\partial q} -\alpha \frac{\partial \beta }{\partial p}
\end{bmatrix} .
\end{align*}
In order to calculate $ \Gamma _{1} =\left[ \Gamma _{1j}^{k}\right]_{j.k\in \{1,2,3,4\}}$ we evaluate
\begin{align*}
\nabla _{q\frac{\partial }{\partial q}}\left( q\frac{\partial }{\partial q}\right) &=q\frac{\partial }{\partial q} +qq\nabla _{\frac{\partial }{\partial q}}\frac{\partial }{\partial q} \\
&=q\frac{\partial }{\partial q} +\frac{qq}{2( \alpha \beta -\gamma \gamma )}\left( \beta \frac{\partial \alpha }{\partial q} -2\gamma \frac{\partial \gamma }{\partial q} +\gamma \frac{\partial \alpha }{\partial p}\right)\frac{\partial }{\partial q} +\frac{qq}{2( \alpha \beta -\gamma \gamma )}\left( -\gamma \frac{\partial \alpha }{\partial q} +2\alpha \frac{\partial \gamma }{\partial q} -\alpha \frac{\partial \alpha }{\partial p}\right)\frac{\partial }{\partial p} ,\\
\nabla _{q\frac{\partial }{\partial q}}\left( p\frac{\partial }{\partial q}\right) &=pq\nabla _{\frac{\partial }{\partial q}}\frac{\partial }{\partial q} \\
&=\frac{pq}{2( \alpha \beta -\gamma \gamma )}\left( \beta \frac{\partial \alpha }{\partial q} -2\gamma \frac{\partial \gamma }{\partial q} +\gamma \frac{\partial \alpha }{\partial p}\right)\frac{\partial }{\partial q} +\frac{pq}{2( \alpha \beta -\gamma \gamma )}\left( -\gamma \frac{\partial \alpha }{\partial q} +2\alpha \frac{\partial \gamma }{\partial q} -\alpha \frac{\partial \alpha }{\partial p}\right)\frac{\partial }{\partial p} ,\\
\nabla _{q\frac{\partial }{\partial q}}\left( q\frac{\partial }{\partial p}\right) &=q\frac{\partial }{\partial p} +qq\nabla _{\frac{\partial }{\partial q}}\frac{\partial }{\partial p} =q\frac{\partial }{\partial p} +\frac{qq}{2( \alpha \beta -\gamma \gamma )}\left( \beta \frac{\partial \alpha }{\partial p} -\gamma \frac{\partial \beta }{\partial q}\right)\frac{\partial }{\partial q} +\frac{qq}{2( \alpha \beta -\gamma \gamma )}\left( -\gamma \frac{\partial \alpha }{\partial p} +\alpha \frac{\partial \beta }{\partial q}\right)\frac{\partial }{\partial p} ,\\
\nabla _{q\frac{\partial }{\partial q}}\left( p\frac{\partial }{\partial p}\right) &=pq\nabla _{\frac{\partial }{\partial q}}\frac{\partial }{\partial p} =\frac{pq}{2( \alpha \beta -\gamma \gamma )}\left( \beta \frac{\partial \alpha }{\partial p} -\gamma \frac{\partial \beta }{\partial q}\right)\frac{\partial }{\partial q} +\frac{pq}{2( \alpha \beta -\gamma \gamma )}\left( -\gamma \frac{\partial \alpha }{\partial p} +\alpha \frac{\partial \beta }{\partial q}\right)\frac{\partial }{\partial p} .
\end{align*}
Therefore a representative with entries in $ P$ of the matrix $ \Gamma _{1}$ of Christoffel symbols of $ \nabla $ is
$$ \frac{1}{2( ab-cc)}\left[\begin{smallmatrix}
2( ab-cc) +bX_{1}( a) -2cX_{1}( c) +cX_{3}( a) & bX_{2}( a) -2cX_{2}( c) +cX_{4}( a) & 2( ab-cc) +bX_{3}( a) -cX_{1}( b) & bX_{4}( a) -cX_{2}( a)\\
0 & 0 & 0 & 0\\
-cX_{1}( a) +2aX_{1}( c) -aX_{3}( a) & -cX_{2}( a) +2aX_{2}( c) -aX_{4}( a) & -cX_{3}( a) +aX_{1}( b) & -cX_{4}( a) +aX_{2}( b)\\
0 & 0 & 0 & 0
\end{smallmatrix}\right].$$
In order to determine $ \Gamma _{2} =\left[ \Gamma _{2j}^{k}\right]_{j.k\in \{1,2,3,4\}}$ we expand
\begin{align*}
\nabla _{p\frac{\partial }{\partial q} \ } q\frac{\partial }{\partial q} &=p\frac{\partial }{\partial p} +pq\nabla _{\frac{\partial }{\partial q}}\frac{\partial }{\partial q}\\
&=p\frac{\partial }{\partial q} +\frac{pq}{2( \alpha \beta -\gamma \gamma )}\left( \beta \frac{\partial \alpha }{\partial q} -2\gamma \frac{\partial \gamma }{\partial q} +\gamma \frac{\partial \alpha }{\partial p}\right)\frac{\partial }{\partial q} +\frac{pq}{2( \alpha \beta -\gamma \gamma )}\left( -\gamma \frac{\partial \alpha }{\partial q} +2\alpha \frac{\partial \gamma }{\partial q} -\alpha \frac{\partial \alpha }{\partial p}\right)\frac{\partial }{\partial p} ,\\
\nabla _{p\frac{\partial }{\partial q} \ } p\frac{\partial }{\partial q} &=pp\nabla _{\frac{\partial }{\partial q}}\frac{\partial }{\partial q} \\
&=\frac{pp}{2( \alpha \beta -\gamma \gamma )}\left( \beta \frac{\partial \alpha }{\partial q} -2\gamma \frac{\partial \gamma }{\partial q} +\gamma \frac{\partial \alpha }{\partial p}\right)\frac{\partial }{\partial q} +\frac{pp}{2( \alpha \beta -\gamma \gamma )}\left( -\gamma \frac{\partial \alpha }{\partial q} +2\alpha \frac{\partial \gamma }{\partial q} -\alpha \frac{\partial \alpha }{\partial p}\right)\frac{\partial }{\partial p} ,\\
\nabla _{p\frac{\partial }{\partial q} \ } q\frac{\partial }{\partial p} &=p\frac{\partial }{\partial p} +pq\nabla _{\frac{\partial }{\partial q}}\frac{\partial }{\partial p} =p\frac{\partial }{\partial p} +\frac{pq}{2( \alpha \beta -\gamma \gamma )}\left( \beta \frac{\partial \alpha }{\partial p} -\gamma \frac{\partial \beta }{\partial q}\right)\frac{\partial }{\partial q} +\frac{pq}{2( \alpha \beta -\gamma \gamma )}\left( -\gamma \frac{\partial \alpha }{\partial p} +\alpha \frac{\partial \beta }{\partial q}\right)\frac{\partial }{\partial p} ,\\
\nabla _{p\frac{\partial }{\partial q} \ } p\frac{\partial }{\partial p} &=pp\nabla _{\frac{\partial }{\partial q}}\frac{\partial }{\partial p} =\frac{pp}{2( \alpha \beta -\gamma \gamma )}\left( \beta \frac{\partial \alpha }{\partial p} -\gamma \frac{\partial \beta }{\partial q}\right)\frac{\partial }{\partial q} +\frac{pp}{2( \alpha \beta -\gamma \gamma )}\left( -\gamma \frac{\partial \alpha }{\partial p} +\alpha \frac{\partial \beta }{\partial q}\right)\frac{\partial }{\partial p} .
\end{align*}
so that a representative with entries in $ P$ of the matrix $ \Gamma _{2}$ is given by
$$ \frac{1}{2( ab-cc)}\left[\begin{smallmatrix}
0 & 0 & 0 & 0\\
2( ab-cc) +bX_{1}( a) -2cX_{1}( c) +cX_{3}( a) & bX_{2}( a) -2cX_{2}( c) +cX_{4}( a) & bX_{3}( a) -cX_{1}( b) & bX_{4}( a) -cX_{2}( a)\\
0 & 0 & 0 & 0\\
-cX_{1}( a) +2aX_{1}( c) -aX_{3}( a) & -cX_{3}( a) +2aX_{3}( c) -aX_{4}( a) & 2( ab-cc) -cX_{3}( a) +aX_{1}( b) & -cX_{4}( a) +aX_{2}( b)
\end{smallmatrix} \right].$$
To determine $ \Gamma _{3} =\left[ \Gamma _{3j}^{k}\right]_{j.k\in \{1,2,3,4\}}$ we investigate
\begin{align*}
\nabla _{q\frac{\partial }{\partial p} \ } q\frac{\partial }{\partial q}&=\frac{qq}{2( \alpha \beta -\gamma \gamma )}\left( \beta \frac{\partial \alpha }{\partial p} -\gamma \frac{\partial \beta }{\partial q}\right)\frac{\partial }{\partial q} +\frac{qq}{2( \alpha \beta -\gamma \gamma )}\left( -\gamma \frac{\partial \alpha }{\partial p} +\alpha \frac{\partial \beta }{\partial q}\right)\frac{\partial }{\partial p} ,\\
\nabla _{q\frac{\partial }{\partial p} \ } p\frac{\partial }{\partial q} &=q\frac{\partial }{\partial q} +\frac{pq}{2( \alpha \beta -\gamma \gamma )}\left( \beta \frac{\partial \alpha }{\partial p} -\gamma \frac{\partial \beta }{\partial q}\right)\frac{\partial }{\partial q} +\frac{pq}{2( \alpha \beta -\gamma \gamma )}\left( -\gamma \frac{\partial \alpha }{\partial p} +\alpha \frac{\partial \beta }{\partial q}\right)\frac{\partial }{\partial p} ,\\
\nabla _{q\frac{\partial }{\partial p} \ } q\frac{\partial }{\partial p} &=\frac{qq}{2( \alpha \beta -\gamma \gamma )}\left( 2\beta \frac{\partial \gamma }{\partial p} -\beta \frac{\partial \beta }{\partial q} -\gamma \frac{\partial \beta }{\partial p}\right)\frac{\partial }{\partial q} +\frac{qq}{2( \alpha \beta -\gamma \gamma )}\left( -2\gamma \frac{\partial \gamma }{\partial p} +\gamma \frac{\partial \beta }{\partial q} -\alpha \frac{\partial \beta }{\partial p}\right)\frac{\partial }{\partial p} ,\\
\nabla _{q\frac{\partial }{\partial p} \ } p\frac{\partial }{\partial p} &=q\frac{\partial }{\partial p} +\frac{pq}{2( \alpha \beta -\gamma \gamma )}\left( 2\beta \frac{\partial \gamma }{\partial p} -\beta \frac{\partial \beta }{\partial q} -\gamma \frac{\partial \beta }{\partial p}\right)\frac{\partial }{\partial q} +\frac{pq}{2( \alpha \beta -\gamma \gamma )}\left( -2\gamma \frac{\partial \gamma }{\partial p} +\gamma \frac{\partial \beta }{\partial q} -\alpha \frac{\partial \beta }{\partial p}\right)\frac{\partial }{\partial p} ,
\end{align*}
so that a representative with entries in $ P$ of the matrix $ \Gamma _{3} =\left[ \Gamma _{3j}^{k}\right]_{j.k\in \{1,2,3,4\}}$ is given by
$$ \frac{1}{2( ab-cc)}\left[\begin{smallmatrix}
bX_{3}( a) -cX_{1}( b) & 2( ab-cc) +bX_{4}( a) -cX_{2}( b) & 2bX_{3}( c) -bX_{1}( b) -cX_{3}( b) & 2bX_{4}( c) -bX_{2}( b) -cX_{4}( b)\\
0 & 0 & 0 & 0\\
-cX_{3}( a) +aX_{1}( b) & -cX_{4}( a) +aX_{2}( b) & -2cX_{3}( c) +cX_{1}( b) -aX_{3}( b) & 2( ab-cc) -2cX_{4}( c) +cX_{2}( b) -aX_{4}( b)\\
0 & 0 & 0 & 0
\end{smallmatrix}\right] .$$
Finally, in order to figure out $ \Gamma _{4} =\left[ \Gamma _{4j}^{k}\right]_{j.k\in \{1,2,3,4\}}$ we evaluate
\begin{align*}
\nabla _{p\frac{\partial }{\partial p} \ } q\frac{\partial }{\partial q}&=\frac{pq}{2( \alpha \beta -\gamma \gamma )}\left( \beta \frac{\partial \alpha }{\partial p} -\gamma \frac{\partial \beta }{\partial q}\right)\frac{\partial }{\partial q} +\frac{pq}{2( \alpha \beta -\gamma \gamma )}\left( -\gamma \frac{\partial \alpha }{\partial p} +\alpha \frac{\partial \beta }{\partial q}\right)\frac{\partial }{\partial p} ,\\
\nabla _{p\frac{\partial }{\partial p} \ } p\frac{\partial }{\partial q} &=p\frac{\partial }{\partial q} +\frac{pp}{2( \alpha \beta -\gamma \gamma )}\left( \beta \frac{\partial \alpha }{\partial p} -\gamma \frac{\partial \beta }{\partial q}\right)\frac{\partial }{\partial q} +\frac{pp}{2( \alpha \beta -\gamma \gamma )}\left( -\gamma \frac{\partial \alpha }{\partial p} +\alpha \frac{\partial \beta }{\partial q}\right)\frac{\partial }{\partial p} ,\\
\nabla _{p\frac{\partial }{\partial p} \ } q\frac{\partial }{\partial p} &=\frac{pq}{2( \alpha \beta -\gamma \gamma )}\left( 2\beta \frac{\partial \gamma }{\partial p} -\beta \frac{\partial \beta }{\partial q} -\gamma \frac{\partial \beta }{\partial p}\right)\frac{\partial }{\partial q} +\frac{pq}{2( \alpha \beta -\gamma \gamma )}\left( -2\gamma \frac{\partial \gamma }{\partial p} +\gamma \frac{\partial \beta }{\partial q} -\alpha \frac{\partial \beta }{\partial p}\right)\frac{\partial }{\partial p} ,\\
\nabla _{p\frac{\partial }{\partial p} \ } p\frac{\partial }{\partial p} &=p\frac{\partial }{\partial p} +\frac{pp}{2( \alpha \beta -\gamma \gamma )}\left( 2\beta \frac{\partial \gamma }{\partial p} -\beta \frac{\partial \beta }{\partial q} -\gamma \frac{\partial \beta }{\partial p}\right)\frac{\partial }{\partial q} +\frac{pp}{2( \alpha \beta -\gamma \gamma )}\left( -2\gamma \frac{\partial \gamma }{\partial p} +\gamma \frac{\partial \beta }{\partial q} -\alpha \frac{\partial \beta }{\partial p}\right)\frac{\partial }{\partial p} ,
\end{align*}
so that a representative with entries in $ P$ of the matrix $ \Gamma _{4}$ is given by
$$ \frac{1}{2( ab-cc)}\left[\begin{smallmatrix}
0 & 0 & 0 & 0\\
bX_{3}( a) -cX_{3}( b) & 2( ab-cc) +bX_{4}( a) -cX_{2}( b) & 2bX_{3}( c) -bX_{1}( b) -cX_{3}( b) & 2bX_{4}( c) -bX_{2}( b) -cX_{4}( b)\\
0 & 0 & 0 & 0\\
-cX_{3}( a) +aX_{3}( b) & -cX_{4}( a) +aX_{2}( b) & -2cX_{3}( b) +cX_{1}( b) -aX_{3}( b) & 2( ab-cc) -2cX_{4}( c) +cX_{2}( b) -aX_{4}( b)
\end{smallmatrix}\right].$$

In the case of the flat invariant metric $g=(\mathrm{d}q)^2+(\mathrm{d}p)^2$ we have that the Levi-Civita connection $\nabla$ coincides with the connection of the adjoint representation (see Subsection \ref{subsec:adjoint}). 

In order to write down the curvature endomorphism we introduce some notation. Let $ \vec{v}$ be the column vector $ \begin{bmatrix}
\frac{\partial \alpha }{\partial q} & \frac{\partial \alpha }{\partial p} & \frac{\partial \beta }{\partial q} & \frac{\partial \beta }{\partial p} & \frac{\partial \gamma }{\partial q} & \frac{\partial \gamma }{\partial p}
\end{bmatrix}^{\top } \in P^{6}$ and evaluate for a matrix $ Q\in P^{6\times 6}$ the corresponding quadratic form $ \vec{v}^{\top } Q\vec{v}$. Let us now put
$$ Q':=\frac{1}{4(\alpha \beta -\gamma \gamma )^{2}}\begin{bmatrix}
0 & 0 & 0 & 0 & 0 & 0\\
0 & \beta  & 0 & 0 & 0 & 0\\
\beta  & -\gamma  & \alpha  & 0 & 0 & 0\\
\gamma  & \alpha  & 0 & 0 & 0 & 0\\
0 & 0 & -2\gamma  & -2\alpha  & 0 & 0 \\
-2\beta  & -2\gamma  & 0 & 0 & 4\gamma  & 0
\end{bmatrix}.$$ 
We used mathematica \cite{Mathematica} to determine
the curvature tensor
$$ \mathcal{R}^{g}\left(\frac{\partial }{\partial q} ,\frac{\partial }{\partial p}\right) =\left(\vec{v}^{\top } Q'\vec{v} -\frac{\left(\frac{\partial ^{2} \alpha }{\partial p\partial p} +\frac{\partial ^{2} \beta }{\partial q\partial q} -2\frac{\partial ^{2} \gamma }{\partial q\partial p}\right)}{2( \alpha \beta -\gamma \gamma )} \right)\begin{bmatrix}
\ \gamma  & \ -\alpha \\
\ \beta  & \ -\gamma 
\end{bmatrix} .$$
From this one can easily deduce expressions for the curvature endomorphism $\mathcal{R}$. For lack of space we only give an example. Namely we have
$$ \mathcal{R}^{g}\left( q\frac{\partial }{\partial q} ,q\frac{\partial }{\partial p}\right) q\frac{\partial }{\partial q} =qqq\mathcal{R}^{g}\left(\frac{\partial }{\partial q} ,\frac{\partial }{\partial p}\right)_{1}^{1}\frac{\partial }{\partial q} \ +qqq\mathcal{R}^{g}\left(\frac{\partial }{\partial q} ,\frac{\partial }{\partial p}\right)_{1}^{2}\frac{\partial }{\partial p} .$$
We deduce
\begin{align*}
\mathcal{R}( X_{1} ,X_{3}) X_{1} &=c \left( \vec{v}_{q}^{\top } Q\vec{v}_{q} -\frac{X_{3}( X_{3}( a)) +X_{1}( X_{1}( b)) -2X_{3}(X_{1}( c))}{2( ab-cc)}\right) X_{1}\\
&\ \ \ \ \ \ \ \ \ \ \ \ \ \ \ \ \ \ \ \ \ \ \ \ \ \ \ \ \ \ \ \ \ \ \ \ \ \ \ \ \ +b \left(\vec{v}_{q}^{\top } Q\vec{v}_{q} -\frac{X_{3}( X_{3}( a)) +X_{1}( X_{1}( b)) -2X_{3}(X_{1}( c))}{2( ab-cc)}\right) X_{3} ,
\end{align*}
where $ \vec{v}_{q} :=\begin{bmatrix}
X_{1}( a) & X_{3}( a) & X_{1}( b) & X_{3}( b) & X_{1}( c) & X_{3}( c)
\end{bmatrix}^{\top }$ and $ Q$ is the matrix $ Q'$ after the substitution $ \alpha \mapsto a,\ \beta \mapsto b,\ \gamma \mapsto c$. The expressions we deduce this way should be understood modulo $ I$ and are unique up to syzygies.

By construction, the curvature endomorphism $\mathcal R$ of $\nabla$ satisfies the \emph{first Bianchi identity}
$$\mathcal{R}(X,Y)Z+\mathcal{R}(Y,Z)X+\mathcal{R}(Z,X)Y=0$$
for all $X,Y,Z\in \operatorname{Der}(A)$ since the curvature tensor $\mathcal{R}^g$ of $\nabla^g$ fulfills it. We mention that the usual proof of the first Bianchi identity for $\nabla$ cannot be used since normal coordinates do not appear to be available. 

Finally, we mention that the formulas above can be used as well to deduce Levi-Civita connection and curvature endomorphism for the Lie-Rinehart algebra $(\operatorname{Der}(\mathcal{C}^\infty(\mathcal{X})),\mathcal{C}^\infty(\mathcal{X}))$ for the simple cone $\mathcal{X}$ seen as a differential space $(\mathcal{X},\mathcal{C}^\infty(\mathcal{X}))$ in the sense of Sikorski. This is because 
the vector fields \eqref{eq:derconegens} can serve as representatives of generators of the $\mathcal{C}^\infty(\mathcal{X})$-module $\operatorname{Der}(\mathcal{C}^\infty(\mathcal{X}))$ as well while the syzygies do not change when replacing real polynomials by smooth function (see \cite[Section 6]{HOS}).


\section{Levi-Civita connections on the orbit space of the Coxeter group $A_{2}$}\label{sec:Dubrovin}

We discuss the orbit space of the Coxeter group $A_{n-1}$, $n\geq 3$, in the real case with emphasis on the simplest special case $A_{2}$. The aim is to elaborate Levi-Civita connections on the orbit space seen as a differential space in the sense of Sikorski (see, e.g. \cite{Navarro, Sniatycki, HOS}). Boris Dubrovin \cite{Dubrovin} has constructed Frobenius manifold structures on the principal stratum of orbit spaces by Coxeter groups. The question if such a Frobenius structure extends to the whole orbit space is open.

The group in the case $A_{n-1}$ is the symmetric group $\Sigma_n=\operatorname{Aut}(\{1,\dots,n\})$.
We let $\sigma \in \Sigma _{n}$ act on $\boldsymbol{x} =\left[ x^{1} \ \dotsc \ x^{n}\right]^{\top } \in \mathbb{R}^{n}$ by 
$$\sigma \left[ x^{1} \ \dotsc \ x^{n}\right]^{\top } =\left[ x^{\sigma ^{-1}( 1)} \ \dotsc \ x^{\sigma ^{-1}( n)}\right]^\top.$$ 
The action preserves the subspace $V:=\ker\left[
1 \  1 \  \cdots  \  1
\right] =\left\{\boldsymbol{x} \in \mathbb{R}^{n} \mid \sum^n_{j=1} x^{j} =0\right\}$.
The $ \Sigma _{n}$-action on $ V$ is coregular, i.e., the categorical quotient $ V/\!\!/\Sigma _{n} =\operatorname{Spec}\left(\mathbb{R[} V]^{\Sigma _{n}}\right)$ is smooth (it is actually $ \simeq \mathbb{R}^{n-1}$). A fundamental system of $ \Sigma _{n}$-invariants for the $ \Sigma _{n}$-action on $ \mathbb{R}^{n}$ is provided by the elementary symmetric functions\footnote{We apologize to the reader that the use of upper indicies is a nuisance here. To avoid misunderstandings we point out that there are no scalars $u$ or $x$ used in this section.}
$$ u^{j}(\boldsymbol{x}) =\sum _{1\leq i_{1} < i_{2} < \dotsc < i_{j} \leq n} x^{i_{1}} x^{i_{2}} \cdots x^{i_{j}} ,$$
which are algebraically independent for $ j=1,\dotsc ,n$. The orbit space $ \mathbb{R}^{n} /\Sigma _{n}$ is mapped to the semialgebraic set $ \boldsymbol{u}\left(\mathbb{R}^{n}\right)$ by the Hilbert map $ \boldsymbol{u} =\left( u^{1} ,\dotsc ,u^{n}\right) :\mathbb{R}^{n}\rightarrow \mathbb{R}^{n}$. Note that $ \boldsymbol{x} \in V$ if and only $ u^{1}(\boldsymbol{x}) =0$ so that
$ u^{2} ,u^{3} ,\dotsc ,u{^{n}}$ is a fundamental system of polynomial $\Sigma_n$-invariants on $ V$. The orbit space $ \mathbb{R}^{n} /\Sigma _{n}$ is cut out from $ \mathbb{R}^{n} /\!\!/\Sigma _{n} \simeq \mathbb{R}^{n}$ by the requirement that the \emph{Bezoutiant}, i.e., the Hankel matrix $ [ p_{i+j}]_{i,j=0,\dotsc ,n-1}$ of the sequence $ ( p_{i})_{i\geq 0}$ in $ \mathbb{R}\left[ x^{1} ,\dotsc ,x^{n}\right]^{\Sigma _{n}}$ of power sums $ p_{j}(\boldsymbol{x}) =\sum _{m=1}^{n} x{^{m}}^{j}$, $j\geq 0$: 
$$ \operatorname{Bez} =\begin{bmatrix}
p_{0} & p_{1} & p_{2} & \cdots  & p_{n-1}\\
p_{1} & p_{2} & p_{3} &  & p_{n}\\
p_{2} & p_{3} & p_{4} &  & p_{n+1}\\
\cdots  &  &  &  & \cdots \\
p_{n-1} & p_{n} & p_{n+1} & \cdots  & p_{2n-2}
\end{bmatrix}$$
is positively semidefinite (see \cite{SchwarzProcesi}). Note that $ p_{1}{}_{|V} =u^{1}{}_{|V} =0$. By \cite[Proposition 2.1]{HHSminimal} this is equivalent to the requirement that for each $ j=2,\dotsc ,n$
$$ \sum _{\boldsymbol{\boldsymbol{\mu} } \in \binom{[ n]}{j}} v_{\boldsymbol{\mu} }(\boldsymbol{u}) \geq 0,$$
where we sum over the invariants $ v_{\boldsymbol{\mu }}(\boldsymbol{u}) :=\mathcal{V}(\boldsymbol{x}_{\boldsymbol{\mu }})^{2}$. Here $ \mathcal{V}(\boldsymbol{x}_{\boldsymbol{\mu }})$ is the Vandermonde determinant of the variables $ x^{i_{1}} ,x^{i_{2}} ,\dotsc ,x^{i_{j}}$ when $ \boldsymbol{\mu} =\{i_{1} < i_{2} < \dotsc < i_{j}\}$, which we interprete as an element in the set  $\binom{[n]}{j}$ of subsets of $[n]=\{1,\dots,n\}$ of cardinality $j$. We write
 $$ \mathcal{X} =\left\{\boldsymbol{u}' =\left( u^{2} ,\dotsc ,u^{n}\right) \in \mathbb{R}^{n-1} \mid \sum _{\boldsymbol{\boldsymbol{\mu }} \in \binom{[ n]}{j}} v_{\boldsymbol{\mu} }(0,\boldsymbol{u}') \geq 0\right\}$$
and note that $ \{0\} \times \mathcal{X}$ is nothing but the image $ \boldsymbol{u}( V)$ of $V$ under the Hilbert embedding $ \boldsymbol{u} =\left( u^{1} ,\dotsc ,u^{n}\right) :\mathbb{R}^{n}\rightarrow \mathbb{R}^{n}$. We attach to the variable $ u^{j}$ the internal degree $ j=\operatorname{deg}\left( u^{j}\right)$ of the corresponding elementary symmetric polynomial $ u^{j}$.

The closures of nontrivial isotropy type strata of the $ \Sigma _{n}$-action on $ \mathbb{R}^{n}$ are parametrized by partitions as follows. The closure of the isotropy stratum (in the classical topology) associated to the partition $ \boldsymbol{\lambda } =( \lambda _{1} ,\lambda _{2} ,\dotsc )$ \ of $ n$ is given by the equations $ ( E_{\boldsymbol{\lambda }})$
\begin{align*}
( E_{1}) \ \ \ \ \ \ \ \ \ \ \ \ \ \ \ \ &x^{1} =\dotsc =x^{\lambda _{1}} ,\\
( E_{2}) \ \ \ \ \ \ \ \ \ \ \ \ \ \ \ \ &x^{\lambda _{1} +1} =\dotsc =x^{\lambda _{1} +\lambda _{2}} ,\\
&\dotsc \\
( E_{\operatorname{length}(\boldsymbol{\lambda }) -1}) \ \ \ &x^{\lambda _{1} +\dotsc +\lambda _{n-2} +1} =\dotsc =x^{\lambda _{1} +\dotsc +\lambda _{n-1}} ,\\
( E_{\operatorname{length}(\boldsymbol{\lambda })}) \ \ \ \ \ \ &x^{\lambda _{1} +\dotsc +\lambda _{n-1} +1} =\dotsc =x^{n} .
\end{align*}
By \cite[Theorem 2.5]{Bierstone} the closures of the isotropy type strata are in one-to-one correspondence with the closures of the minimal strata in $ \mathcal{X}$. The corresponding isotropy groups are conjugate to 
$ \Sigma _{\lambda _{1}} \times \Sigma _{\lambda _{2}} \times \cdots \times \Sigma _{\lambda _{n}}$. 

In principle, the equations for the closures of the minimal strata can be determined by elimination. We have that $ ( E_{\boldsymbol{\lambda }})$ implies $ ( E_{\boldsymbol{\lambda } '})$ 
if and only if $ \boldsymbol{\lambda } '$ is a \textit{refinement} of $ \boldsymbol{\lambda }$ and that the equation $ E_{( 1,1,\dotsc ,1)}$ is empty. The restriction $ ( E_{\boldsymbol{\lambda }})_{|V}$ of $ ( E_{\boldsymbol{\lambda }})$ to $ V$ can be understood as omitting $ ( E_{( n)})$ since these equation vanishes on $ V$. The equation $ ( E_{( 2,1,\dotsc ,1)})_{|V}$ is given by the vanishing of the \textit{discriminant} of the polynomial 
$$ f( t) =t^{n+1} +\sum_{j=1}^{n-1}( -1)^{n-j} \ u^{n+1-j} \ t^{j} =\prod _{l=1}^{n+1}\left( t-x^{l}\right) .$$
The polynomial $ f( t)$ is the also known as the \textit{miniversal unfolding of} $ t^{n+1}$ (cf. \cite{Dubrovin}, \cite[Section 1]{Bierstone}). 

Before we look into an example let us recall that in order to express the power sums in terms of $ \boldsymbol{u}$ we may expand
$$ \log\left(\sum _{j\geq 0} u^{j} s^{j}\right) =\sum _{m\geq 1}\frac{( -1)^{m+1}}{m} p_{m} s^{m} \in \mathbb{R}\left[ x^{1} ,\dotsc ,x^{n}\right]^{\Sigma _{n}} \llbracket s\rrbracket .$$
We deduce
\begin{align}\label{eq:Newton}
p_{1} &=u^{1} ,\\
\nonumber p_{2} &=-2u^{2} +u{^{1}}^{2} ,\\
 \nonumber p_{3} &=3u^{3} -3u^{1} u^{2} +u{^{1}}^{3} ,\\
\nonumber p_{4} &=-4u^{4} +2u{^{2}}^{2} +4u^{1} u^{3} -4u{^{1}}^{2} u^{2} +u{^{1}}^{4} ,\ \text{etc.} ,
\end{align}
which, after restricting to $ V$, simplifies to
$ p_{2} =-2u^{2} ,\ p_{3} =3u^{3} ,\ p_{4} =-4u^{4} +2u{^{2}}^{2}$, etc. 

We now specialize to the case $ n=3$. The reader is invited to elaborate more complicated cases.
The partition $ ( 1,1,1)$ corresponds to the closure of the principal stratum, which has trivial isotropy group. The partition $ ( 2,1)$, on the other hand, corresponds to the closure of the stratum with isotropy group $ \Sigma _{2}$, which is given by the vanishing discriminant $ \Delta :=4u{^{2}}^{3}+27u{^{3}}^{2}$
$$( E_{2}) \ \ \ \ \ \ \ \ \ \ \Delta =0.$$
The equation is homogeneous of internal degree $ 6$. We collect the data of the stratification in Table \ref{tab:Sigma3}.
\begin{table}[]
    \centering
  $$  \begin{array}{ c||c|c|c }
\text{isotropy group} & \{e\} & \Sigma _{2}\times\{e\} & \Sigma _{3}\\
\hline
\text{partition} & ( 1,1,1) & ( 2,1) & \\
\hline
\text{equation of stratum closure} & \text{empty} & \Delta =4u{^{2}}^{3} +27u{^{3}}^{2}=0 & u^{1} =u^{2} =u^{3} =0
\end{array}$$
    \caption{Strata of $V/\Sigma_3$.}
    \label{tab:Sigma3}
\end{table}

The requirement of the positive semidefiniteness of the restriction $ \operatorname{Bez}_{|V}$ of the Bezoutiant $ \operatorname{Bez}$ to $ V$ is expressed by
\begin{align*}
\begin{vmatrix}
p_{0} & 0\\
0 & p_{2}
\end{vmatrix} &=\begin{vmatrix}
3 & 0\\
0 & p_{2}
\end{vmatrix} =3p_{2} =-6u^{2} \geq 0,\\
\left| \operatorname{Bez}_{|V}\right| &=\begin{vmatrix}
p_{0} & 0 & p_{2}\\
0 & p_{2} & p_{3}\\
p_{2} & p_{3} & p_{4}
\end{vmatrix} =\begin{vmatrix}
3 & 0 & p_{2}\\
0 & p_{2} & p_{3}\\
p_{2} & p_{3} & p_{4}
\end{vmatrix}
=3p_{2} p_{4} -p_{2}^{3} -3p_{3}^{2} =-6u^{2}\left( -4u^{4} +2u{^{2}}^{2}\right) -\left( -2u^{2}\right)^{3} -27u{^{3}}^{2}\\
&=-4u{^{2}}^{3} -27u{^{3}}^{2} \geq 0,
\end{align*}
where we keep in mind that $ u^{4} =0$ for $ n=3$. Since $\Delta\leq 0$ implies $u^2\leq 0$ we have that 
$$ \mathcal{X} =\left\{\left( u^{2} ,u^{3}\right) \in \mathbb{R}^{2} \mid \Delta =4u{^{2}}^{3} +27u{^{3}}^{2}\leq 0\right\},$$
which is depicted in Figure \ref{fig:discr}.
\begin{figure}
    \centering
    \includegraphics[width=0.3\linewidth]{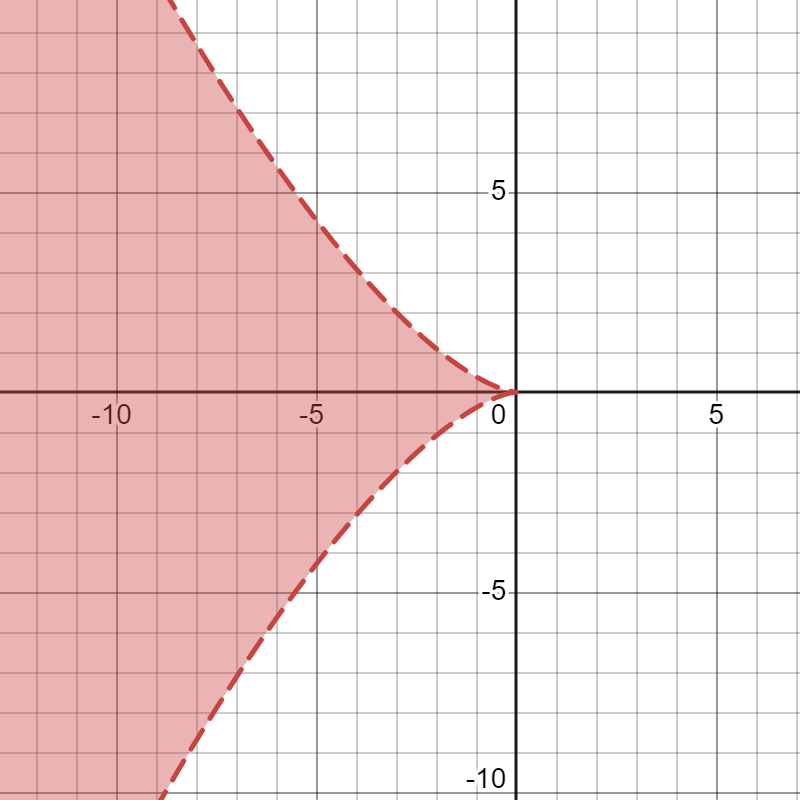}
    \caption{Hilbert embedding $\mathcal{X}$ of $V/\Sigma_3$.}
    \label{fig:discr}
\end{figure}

To determine generators of the $ \mathbb{R}\left[ x^{1} ,x^{2} ,x^{3}\right]^{\Sigma _{3}}$-module $ \operatorname{Der}\left(\mathbb{R}\left[ x^{1} ,x^{2} ,x^{3}\right]\right)^{\Sigma _{3}}$ we used Magma \cite{Magma} according to the recipe \cite[Subsection 4.2.3]{Derksen}. They are\footnote{This can be done by picking invariants from the Magma output of
{\fontfamily{pcr}\selectfont K := RationalField();G := PermutationGroup< 6 |(1,2)( 5,4),(2, 3)(6,5)>; R := InvariantRing(G, K);FundamentalInvariants(R);}
that are of degree $ 1$ in the variables \ {\fontfamily{pcr}\selectfont x4,x5,x6}.}
\begin{align*}
\frac{\partial }{\partial x^{1}} +\frac{\partial }{\partial x^{2}} +\frac{\partial }{\partial x^{3}} ,\ x^{1}\frac{\partial }{\partial x^{1}} +x^{2}\frac{\partial }{\partial x^{2}} +x^{3}\frac{\partial }{\partial x^{3}} ,\ x{^{1}}^{2}\frac{\partial }{\partial x^{1}} +x{^{2}}^{2}\frac{\partial }{\partial x^{2}} +x{^{3}}^{2}\frac{\partial }{\partial x^{3}} .
\end{align*}
To write them in terms of the Hilbert embedding we elaborate their action on elementary symmetric functions
$$ \begin{array}{ c|c c c }
 & u^{1} =x^{1} +x^{2} +x^{3} & u^{2} =x^{1} x^{2} +x^{1} x^{3} +x^{2} x^{3} & u^{3} =x^{1} x^{2} x^{3}\\
\hline
\frac{\partial }{\partial x^{1}} +\frac{\partial }{\partial x^{2}} +\frac{\partial }{\partial x^{3}} & 3 & 2u^{1} & u^{2}\\
x^{1}\frac{\partial }{\partial x^{1}} +x^{2}\frac{\partial }{\partial x^{2}} +x^{3}\frac{\partial }{\partial x^{3}} & u^{1} & 2u^{2} & 3u^{3}\\
x{^{1}}^{2}\frac{\partial }{\partial x^{1}} +x{^{2}}^{2}\frac{\partial }{\partial x^{2}} +x{^{3}}^{2}\frac{\partial }{\partial x^{3}} & u{^{1}}^{2} -2u^{2} & -2u^{1} u^{2} +u{^{1}}^{3} -3u^{3} & u^{1} u^{3}
\end{array}$$
using the relations \eqref{eq:Newton}.  Hence the three vector fields are mapped under the Hilbert embedding to
\begin{align*}
\frac{\partial }{\partial x^{1}} +\frac{\partial }{\partial x^{2}} +\frac{\partial }{\partial x^{3}} \ &\mapsto \xi _{1} :=3\frac{\partial }{\partial u^{1}} +2u^{1}\frac{\partial }{\partial u^{2}} +u^{2}\frac{\partial }{\partial u^{3}} ,\\
x^{1}\frac{\partial }{\partial x^{1}} +x^{2}\frac{\partial }{\partial x^{2}} +x^{3}\frac{\partial }{\partial x^{3}} \ &\mapsto \xi _{2} :=u^{1}\frac{\partial }{\partial u^{1}} +2u^{2}\frac{\partial }{\partial u^{2}} +3u^{3}\frac{\partial }{\partial u^{3}} ,\\
x{^{1}}^{2}\frac{\partial }{\partial x^{1}} +x{^{2}}^{2}\frac{\partial }{\partial x^{2}} +x{^{3}}^{2}\frac{\partial }{\partial x^{3}} \ &\mapsto \xi _{3} :=\left( u{^{1}}^{2} -2u^{2}\right)\frac{\partial }{\partial u^{1}} +\left( -2u^{1} u^{2} +u{^{1}}^{3} -3u^{3}\right)\frac{\partial }{\partial u^{2}} +u^{1} u^{3}\frac{\partial }{\partial u^{3}} .
\end{align*}
Note that this map is compatible with the brackets. The generators $ \xi _{1} ,\xi _{2} ,\xi _{3}$ fulfill the commutation relations
$$ \begin{array}{ c|c c c }
[ \ ,\ ] & \xi _{1} & \xi _{2} & \xi _{3}\\
\hline
\xi _{1} & 0 & \xi _{1} & 2\xi _{2}\\
\xi _{2} & -\xi _{1} & 0 & \xi _{3}\\
\xi _{3} & -2\xi _{2} & -\xi _{3} & 0
\end{array} .$$

We are actually aiming at determining generators of the $\mathbb{R}[ V\mathbb{]}^{\Sigma _{3}}$-module $ \operatorname{Der}\left(\mathbb{R}[ V\mathbb{]}^{\Sigma _{3}}\right)$. To this end we will use a slight generalization of \cite[Proposition 4.1]{HOS}, which can be proven in an analogous fashion.

\begin{proposition}
Let $ ( L,P)$ a Lie-Rinehart algebra with anchor $ \alpha :L\rightarrow \operatorname{Der}( P)$ and let $ I=( f_{1} ,\dotsc ,f_{k})$ be an ideal in $ P$. Define the $P$-sumodule $ L_{I} :=\{X\in L\mid X( I) \subseteq I\}$ of elements of $L$ tangent to $I$ and assume that the $ P$-module $ L$ is generated by $ \xi _{1} ,\dotsc ,\xi _{m}$. Define
$$ M=\begin{bmatrix}
f_{1} & f_{2} & \cdots  & f_{k} & \xi _{1}( f_{1}) & \xi _{2}( f_{1}) & \cdots  & \xi _{m}( f_{1})\\
f_{1} & f_{2} & \cdots  & f_{k} & \xi _{1}( f_{2}) & \xi _{2}( f_{2}) & \cdots  & \xi _{m}( f_{2})\\
\cdots  &  &  &  &  &  &  & \cdots \\
f_{1} & f_{2} & \cdots  & f_{k} & \xi _{1}( f_{k}) & \xi _{2}( f_{k}) & \cdots  & \xi _{m}( f_{k})
\end{bmatrix} \in P^{k\times ( k+m)}$$
and assume that $ \begin{bmatrix}
g^{1} & \cdots  & g^{k} & h^{1} & \cdots  & h^{m}
\end{bmatrix} ^\top\in \ker M$. Then 
$$ L_{I} =\left\{\sum _{j} h^{j} \xi _{j} \mid \begin{bmatrix}
g^{1} & \cdots  & g^{k} & h^{1} & \cdots  & h^{m}
\end{bmatrix}^\top \in \ker M\right\} +\ker \alpha .$$
\end{proposition}
We now apply the proposition to the principal ideal $I=(u^1)$ in $\mathbb{R}[u^1,u^2,u^3]$ and the $\mathbb{R}[u^1,u^2,u^3]$-submodule $L$ of $\operatorname{Der}(\mathbb{R}[u^1,u^2,u^3])$ generated by $ \xi _{1},\xi _{2} ,\xi _{3}$.
We used Macaulay2 \cite{M2} to calculate\footnote{Here the Macaulay2 code is simply {\fontfamily{pcr}\selectfont R=QQ[u1,u2,u3]; M=matrix\{\{u1,3,u1,u1*u1-2*u2\}\}; ker M.}} the kernel of
$$M=\begin{bmatrix}
u^{1} & 3 & u^{1} & u{^{1}}^{2} -2u^{2}
\end{bmatrix} \in \mathbb{R}\left[ u^{1} ,u^{2} ,u^{3}\right]^{1\times 4}$$
to be the image of
$$ \begin{bmatrix}
3 & 0 & 0\\
-u^{1} & -u^{1} & -u{^{1}}^{2} +2u^{2}\\
0 & 3 & 0\\
0 & 0 & 3
\end{bmatrix}.$$
This means that $ L_{I} /IL$ is the $ P/I$-module generated by 
$ \xi_2 +IL,\ 2u^{2} \xi _{1} +3\xi _{3} +IL$.
Note that $P/I\simeq \mathbb{R}\left[ u^{2} ,u^{3}\right]$ and that under this isomorphism we have
\begin{align*}
\xi _{2} +IL\in L_{I} /IL\ \ &\mapsto \zeta_1:=2u^{2}\frac{\partial }{\partial u^{2}} +3u^{3}\frac{\partial }{\partial u^{3}} ,\\
\ 2u^{2} \xi _{1} +3\xi _{3} +IL\in L_{I} /IL\ &\mapsto \zeta_2:=-9u^{3}\frac{\partial }{\partial u^{2}} +2u{^{2}}^{2}\frac{\partial }{\partial u^{3}} .
\end{align*}
Moreover, since 
\begin{align*}
\left[ \xi _{2} ,2u^{2} \xi _{1} +3\xi _{3}\right] &=2u^{2}[ \xi _{2} ,\xi _{1}] +3[ \xi _{2} ,\xi _{3}] =\overbrace{2\xi _{2}\left( u^{2}\right) \xi _{1}}^{=4u^{2} \xi _{1}} -2u^{2} \xi _{1} +3[ \xi _{2} ,\xi _{3}] =2u^{2} \xi _{1} +3\xi _{3} ,\\
\left[ \zeta_1 ,\zeta_2\right]&=18u^{3}\frac{\partial }{\partial u^{2}} +8u{^{2}}^{2}\frac{\partial }{\partial u^{3}} -27u^{3}\frac{\partial }{\partial u^{2}} -6u{^{2}}^{2}\frac{\partial }{\partial u^{3}}=-9u^{3}\frac{\partial }{\partial u^{2}} +2u{^{2}}^{2}\frac{\partial }{\partial u^{3}}=\zeta_2.
\end{align*}
this map is compatible with the brackets\footnote{The $\mathcal{C}^\infty(\mathbb{R}^2)$-span of $\zeta_1,\zeta_2$ may be interpreted as a singular foliation.}.

We say that $ X\in \operatorname{Der}\left(\mathcal{C}^{\infty }(\mathcal{X})\right)$ \textit{preserves strata} if for each stratum closure $ \overline{S}$ with vanishing ideal $ \mathcal{I}_{S}$
we have $ X(\mathcal{I}_{S}) \subseteq \mathcal{I}_{S}$. The set of such derivations we denote by $ \mathfrak{X}^\infty(\mathcal{X})$. Under certain largeness assumptions on the representation of the group one can show that $ \operatorname{Der}\left(\mathcal{C}^{\infty }(\mathcal{X})\right) =\mathfrak{X}^\infty(\mathcal{X})$, see \cite{liftingHomo}. 
This is for example the case for the example of the single cone of the previous section.
For coregular representations, however, the inclusion is proper $ \mathfrak{X}^\infty(\mathcal{X}) \subset \operatorname{Der}\left(\mathcal{C}^{\infty }(\mathcal{X})\right)$. In fact, $ \left(\mathfrak{X}^\infty(\mathcal{X}) ,\mathcal{C}^{\infty }(\mathcal{X})\right)$ forms a Lie-Rinehart subalgebra of $ \left(\operatorname{Der}\left(\mathcal{C}^{\infty }(\mathcal{X})\right) ,\mathcal{C}^{\infty }(\mathcal{X})\right)$.
In a similar way we define the Lie Rinehart algebra $(\mathfrak{X}^\infty(V/\Sigma_3),\mathcal{C}^\infty(V/\Sigma_3))$ of stata preserving derivations.
\begin{proposition}
The two vector fields 
$ \zeta _{1} =2u^{2}\frac{\partial }{\partial u^{2}} +3u^{3}\frac{\partial }{\partial u^{3}} ,\ \zeta _{2} =-9u^{3}\frac{\partial }{\partial u^{2}} +2u{^{2}}^{2}\frac{\partial }{\partial u^{3}}$
form a free system of generators of the $ \mathcal{C}^{\infty }(\mathcal{X})$-module
$ \mathfrak{X}^\infty(\mathcal{X})$. The pull-back of the composition of the map $ \boldsymbol{u}_{|V}$ with the projection $ \mathbb{R}^{3}\rightarrow \mathbb{R}^{2} ,\ \left( u^{1} ,u^{2} ,u^{3}\right) \mapsto \ \left( u^{2} ,u^{3}\right)$ induces an isomorphism of Lie-Rinehart algebras $ \left( \mathfrak{X}^\infty(\mathcal{X}) ,\mathcal{C}^{\infty }(\mathcal{X})\right)$ and
$ \left(\mathfrak{X}^\infty( V/\Sigma _{3}) ,\mathcal{C}^{\infty }( V/\Sigma _{3})\right)$.
\end{proposition}

\begin{proof}
The $ \Sigma _{3}$-module $ \mathbb{R}^{3}$ decomposes into direct sum of a trivial $ \Sigma _{3}$-module $ \boldsymbol{k}$ and the irreducible $ \Sigma _{3}$-module $ V$. 
Our first observation is that $ \zeta _{1} ,\zeta _{2}$ preserve principal ideal $ ( \Delta )$ and the inequality $ \Delta \leq 0$:
$$ \begin{array}{ c|c }
 & \Delta =4u{^{2}}^{3} +27u{^{3}}^{2}\\
\hline
\zeta _{1} =2u^{2}\frac{\partial }{\partial u^{2}} +3u^{3}\frac{\partial }{\partial u^{3}} & 6\left( 4u{^{2}}^{3} +27u{^{3}}^{2}\right) =6\Delta \\
\zeta _{2} =-9u^{3}\frac{\partial }{\partial u^{2}} +2u{^{2}}^{2}\frac{\partial }{\partial u^{3}} & -108u{^{2}}^{2} u^{3} +108u{^{2}}^{2} u^{3} =0
\end{array}.$$
Moreover, $ \zeta _{1} ,\zeta _{2}$ preserve the vanishing ideal of the origin $0\in\mathbb{R}^2$. In fact, the image of the canonical map
$$\operatorname{Der}\left(\mathbb{R}[V]\right)^{\Sigma _{3}}\xrightarrow{\lambda }\operatorname{Der}(\mathbb{R}[V]^{\Sigma _{3}})$$
is the $ \mathbb{R}[V]^{\Sigma_3}$-module $\mathfrak{X}(V/\Sigma_3)$  of strata preserving regular vector fields on $V/\Sigma_3$, see \cite{liftingHomo}. Recall  that 
$ \boldsymbol{u}_{|V}^{*} :\mathbb{R}[V]^{\Sigma _{3}}\rightarrow \mathbb{R}[\mathcal{X}]$ is an isomorphism of $ \mathbb{R}$-algebras and that $\boldsymbol{u}_{|V}$ maps strata to strata in a one-to-one fashion.
So we have an isomorphism  of Lie-Rinehart algebras over $\boldsymbol{u}_{|V}^{*}$ from $(\mathfrak{X}(V/\Sigma_3),\mathbb{R}[V]^{\Sigma_3})$ to the Lie-Rinehart algebra $(\mathfrak{X}(\mathcal{X}),\mathbb{R}[\mathcal{X}])$ of strata preserving regular vector fields on $\mathcal{X}$.
So by construction $ \zeta _{1} ,\zeta _{2}$ generate  the $ \mathbb{R}[\mathcal{X}]$-module $\mathfrak{X}(\mathcal{X})$ of strata preserving regular vector fields on $\mathcal{X}$. 
Recall that by \cite[Lemma 3.7]{Bierstone}, \cite{liftingHomo} the morphism of Lie-Rinehart algebras
$$\operatorname{Der}\left(\mathcal{C}^{\infty }( V)\right)^{\Sigma _{3}}\xrightarrow{\Lambda }\mathfrak{X}^{\infty }( V/\Sigma _{3})$$ is onto. Here it is important to remember that by Schwarz's Theorem on smooth invariants \cite{Schwarzdiffinv}
$ \boldsymbol{u}_{|V}^{*} :\mathcal{C}^{\infty }(V)^{\Sigma _{3}}\rightarrow \mathcal{C}^{\infty }(\mathcal{X})$ is an isomorphism of $ \mathbb{R}$-algebras.
So we have an isomorphism $\mathfrak{X}^\infty(V/\Sigma_3)\to\mathfrak{X}^\infty(\mathcal{X})$ of Lie-Rinehart algebras over $\boldsymbol{u}_{|V}^{*}$.

It remains to show that the generators of the $ \mathbb{R}[ V]^{\Sigma _{3}}$-module $ \operatorname{Der}(\mathbb{R}[ V])^{\Sigma _{3}}$ generate the $ \mathcal{C}^{\infty }( V)^{\Sigma _{3}}$-module $ \operatorname{Der}\left(\mathcal{C}^{\infty }(V)\right)^{\Sigma _{3}}$. This follows from the following general consideration. Let $ G$ be a compact Lie group and $ X$ a smooth $ G$-manifold and $ U$ a finite dimensional $ G$-module. We denote the $ \mathcal{C}^{\infty }( X)^{G}$-module of smooth $ G$-equivariant maps $ \phi :X\rightarrow U$ by $ \operatorname{Map}( X,U)^{G}$. In analogy to \cite[Subsection 4.2.3]{Derksen} the generators for this $ \mathcal{C}^{\infty }( X)^{G}$-module consist of those algebra generators of smooth of the algebra of $ G$-invariant smooth functions on $ U^{*} \times X$ that are linear in $ U^{*}$. In other words, if we denote by $ \mathfrak{M}_{( X,U)}$ the ideal of $ f\in \mathcal{C}^{\infty }\left( U^{*} \times X\right)$ such that $ f_{|\{0\} \times X} =0$ then we are looking for the generators of the  $ \mathcal{C}^{\infty }( X)^{G}$-module $ \left(\mathfrak{M}_{( X,U)} /\mathfrak{M}_{( X,U)}^{2}\right)^{G}$. If now $ V$ is a finite dimensional $ G$-module we have that $ \operatorname{Der}\left(\mathcal{C}^{\infty }( V)\right)^G$ is isomorphic as a $ G$-module to $ \operatorname{Map}( V,V)^G$ so that $ \operatorname{Der}\left(\mathcal{C}^{\infty }( V)\right)^{G}$ is generated by $ \left(\mathfrak{M}_{( V,V)} /\mathfrak{M}_{( V,V)}^{2}\right)^{G}$. Let $ \psi _{1} ,\dotsc ,\psi _{N} \ $be a fundamental system of polynomial invariants on $ V^{*} \times V$ and $ \boldsymbol{\psi } =( \psi _{1} ,\dotsc ,\psi _{N} \ ) :V^{*} \times V\rightarrow \mathbb{R}^{N}$ the corresponding Hilbert map. By the Schwarz's Theorem of smooth invariants we have that the pullback $ \boldsymbol{\psi }^{*} :\mathcal{C}^{\infty }\left(\mathbb{R}^{N}\right)\rightarrow \mathcal{C}^{\infty }\left( V^{*} \times V\right)^{G}$ is onto and hence $ \mathcal{C}^{\infty }\left( V^{*} \times V\right)^{G}$ is generated as an algebra by $ \psi _{1} ,\dotsc ,\psi _{N}$. On the other hand, let $ \mathfrak{m}_{( V,V)}$ be the ideal of $ f\in \mathbb{R}\left[ V^{*} \times V\right]$ such that $ f_{|\{0\} \times V} =0$. Then the generators of the $ \mathbb{R}[ V\mathbb{]}^{G}$-module $ \operatorname{Mor}( V,V)^{G} \simeq \operatorname{Der}(\mathbb{R}[ V\mathbb{]})^{G}$ can be interpreted as generators of \ the $ \mathbb{R}[ V\mathbb{]}^{G}$-module $ \mathfrak{m}_{( V,V)} /\mathfrak{m}_{( V,V)}^{2}$ and vice versa.
The latter also can be constructed by picking from $ \psi _{1} ,\dotsc ,\psi _{N}$ those linear in $ V^{*}$.
\end{proof}

Next, we study metrics on the orbit space of the Coxeter group $A_2=\Sigma_3$ that arise by symmetry reduction.  Note that the space of quadratic $ \Sigma _{3}$-invariant polynomial functions on $ \mathbb{R}^{3}$ is $ 2$-dimensional. As generators we can take, e.g., $ p_{2}$ and $ u^{2}$ and can interprete these as Riemannian metrics on $ \mathbb{R}^{3}$. Those in turn can be linearly combined over $ \mathcal{C}^{\infty }\left(\mathbb{R}^{3}\right)^{\Sigma _{3}}$ to construct the most general $ \Sigma _{3}$-invariant symmetric bilinear form $ g$ on $ \mathbb{R}^{3}$. Using Gerald Schwarz's theorem on smooth invariants \cite{Schwarzdiffinv} this form can be expressed in terms of the coordinates $ u^{1} ,u^{2} ,u^{3}$:
\begin{align*}
g&:=\rho \operatorname{Bez} +\tau \begin{bmatrix}
p_{0} & p_{1} & p_{2}\\
p_{1} & u^{2} & \frac{u^{1} u^{2} -3u^{3}}{2}\\
p_{2} & \frac{u^{1} u^{2} -3u^{3}}{2} & \frac{p_{2}^{2} -p_{4}}{2}
\end{bmatrix}=\begin{bmatrix}
p_{0}( \rho +\tau ) & p_{1}( \rho +\tau ) & p_{2}( \rho +\tau )\\
p_{1}( \rho +\tau ) & p_{2} \rho +u^{2} \tau  & p_{3} \rho +\frac{u^{1} u^{2} -3u^{3}}{2} \tau \\
p_{2}( \rho +\tau ) & p_{3} \rho +\frac{u^{1} u^{2} -3u^{3}}{2} \tau  & p_{4} \rho +\frac{p_{2}^{2} -p_{4}}{2} \tau 
\end{bmatrix} ,
\end{align*}
where $ \rho=\rho(u^1,u^2,u^3) ,\tau=\tau(u^1,u^2,u^3)$ are smooth functions. The determinant of $g$ is $( \rho +2\tau )( \rho -\tau )^{2}$.
We are aiming at identifying the corresponding symmetric bilinear form on tnagent Lie-Rinehart algebra $(\operatorname{Der}(\mathcal{X}),\mathcal{C}^\infty(\mathcal{X}))$ of the differential space $(\mathcal{X},\mathcal{C}^\infty(\mathcal{X}))$.

\begin{proposition}\label{prop:S3metric}
The most general semi-Riemannian metric for the Lie-Rinehart algebra $ \left(\mathfrak{X}^\infty(\mathcal{X}) ,\mathcal{C^{\infty }( X})\right)$ arising from symmetry reduction of a $ \Sigma _{3}$-invariant
semi-Riemannian metric on $ \mathbb{R}^{3}$ is of the form
$$\begin{array}{ c|c c }
G & \zeta _{1} & \zeta _{2}\\
\hline
 \zeta _{1} & -u^{2} f & \frac{9}{2} u^{3} f\\
 \zeta _{2} & \frac{9}{2} u^{3} f & 3u{^{2}}^{2} f
\end{array} ,$$
where $ f:\mathbb{R}^{2}\rightarrow \mathbb{R} ,\ \left( u^{2} ,u^{3}\right)$ is a smooth nowhere vanishing function. The determinant of $G$ equals $-\frac{3}{4} f^{2}\Delta$.
\end{proposition}
\begin{proof}
    Restricting $ g$ to
$ \operatorname{Der}\left(\mathbb{R}\left[ x^{1} ,x^{2} ,x^{3}\right]\right)^{\Sigma _{3}} \times \operatorname{Der}\left(\mathbb{R}\left[ x^{1} ,x^{2} ,x^{3}\right]\right)^{\Sigma _{3}}$
in terms the generators $ \xi _{1} ,\xi _{2} ,\xi _{3}$ we obtain the metric
$$ G_{3} :=\rho \operatorname{Bez} +\tau \begin{bmatrix}
p_{0} & p_{1} & p_{2}\\
p_{1} & u^{2} & \frac{u^{1} u^{2} -3u^{3}}{2}\\
p_{2} & \frac{u^{1} u^{2} -3u^{3}}{2} & \frac{p_{2}^{2} -p_{4}}{2}
\end{bmatrix} =\begin{bmatrix}
p_{0}( \rho +\tau ) & p_{1}( \rho +\tau ) & p_{2}( \rho +\tau )\\
p_{1}( \rho +\tau ) & p_{2} \rho +u^{2} \tau  & p_{3} \rho +\frac{u^{1} u^{2} -3u^{3}}{2} \tau \\
p_{2}( \rho +\tau ) & p_{3} \rho +\frac{u^{1} u^{2} -3u^{3}}{2} \tau  & p_{4} \rho +\frac{p_{2}^{2} -p_{4}}{2} \tau 
\end{bmatrix}$$
and restricting $ G_{3}$ to $ V$ this simplifies to
\begin{align*}
 G_{3}{}_{|V}& =\rho \begin{bmatrix}
p_{0} & 0 & -2u^{2}\\
0 & -2u^{2} & 3u^{3}\\
-2u^{2} & 3u^{3} & 2u{^{2}}^{2}
\end{bmatrix} +\tau \begin{bmatrix}
p_{0} & 0 & -2u^{2}\\
0 & u^{2} & -\frac{3}{2} u^{3}\\
-2u^{2} & -\frac{3}{2} u^{3} & u{^{2}}^{2}
\end{bmatrix} \\
&=\begin{bmatrix}
p_{0}( \rho +\tau ) & 0 & -2u^{2}( \rho +\tau )\\
0 & -u^{2}( 2\rho -\tau ) & \frac{3}{2} u^{3}( 2\rho -\tau )\\
-2u^{2}( \rho +\tau ) & \frac{3}{2} u^{3}( 2\rho -\tau ) & u{^{2}}^{2}( 2\rho +\tau )
\end{bmatrix} .
\end{align*}
The metric $ G$ in turn is obtained by evaluating $ G_{3}$ on $ \xi _{2} ,\ 2u^{2} \xi _{1} +3\xi _{3}$ and restricting to $V$:
$$ \begin{array}{ c|c c }
G & \xi _{2} & 2u^{2} \xi _{1} +3\xi _{3}\\
\hline
\xi _{2} & -u^{2}( 2\rho -\tau ) & \frac{9}{2} u^{3}( 2\rho -\tau )\\
2u^{2} \xi _{1} +3\xi _{3} & \frac{9}{2} u^{3}( 2\rho -\tau ) & 3u{^{2}}^{2}( 2\rho -\tau )
\end{array}$$
Here we simplified ${G_{3}}_{|V}\left( 2u^{2} \xi _{1} +3\xi _{3} ,2u^{2} \xi _{1} +3\xi _{3}\right) =12u{^{2}}^{2}( \rho +\tau ) +9u{^{2}}^{2}( 2\rho +\tau ) -24u{^{2}}^{2}( \rho +\tau ) =3u{^{2}}^{2}( 2\rho -\tau ) .$
Note that
$ \det( G) =\left( -3u{^{2}}^{3} -\frac{81}{4} u{^{3}}^{2}\right)( 2\rho -\tau )^{2} =-\frac{3}{4} \Delta ( 2\rho -\tau )^{2}$
and we define $ f\left( u^{2} ,u^{3}\right) :=2\rho \left( 0,u^{2} ,u^{3}\right) -\tau \left( 0,u^{2} ,u^{3}\right) .$
\end{proof}

\begin{theorem}
There is a Levi-Civita connection $\nabla $ for the Lie-Rinehart algebra $ \left(\mathfrak{X}^\infty(\mathcal{X}) ,\mathcal{C^{\infty }( X})\right)$ equipped with the semi-Riemannian metric $G$ of Proposition \ref{prop:S3metric}. Its matrices $\Gamma _{i} =\left[ \Gamma _{ij}^{k}\right]_{j,k=1,2}$ of Christoffel symbols $\Gamma _{ij}^{k}$ are
\begin{align*}
\Gamma _{1}&=\begin{bmatrix}
1+u^{2}\frac{\partial (\log f)}{\partial u^2}+3u^{3}\frac{\partial (\log f)}{\partial u^{3}}& u{^{2}}^{2}\frac{\partial (\log f)}{\partial u^{3}}\\
\frac{1}{3} u^{2}\frac{\partial (\log f)}{\partial u^{3}} & 2 +u^{2}\frac{\partial (\log f)}{\partial u^{2}}
\end{bmatrix} ,\ \\
\Gamma _{2}&=\begin{bmatrix}
u{^{2}}^{2}\frac{\partial (\log f)}{\partial u^{3}} & 3u^{2}\left( -2+u^2\frac{\partial (\log f)}{\partial u^{2}}\right)\\
1+u^{2}\frac{\partial \log f}{\partial u^{2}} & -9u^{3}\frac{\partial (\log f)}{\partial u^{2}} +u{^{2}}^{2}\frac{\partial (\log f)}{\partial u^{3}}
\end{bmatrix}
\end{align*}
and its curvature endomorphism 
$$ \mathcal{R}( \zeta _{1} ,\zeta _{2}) =\begin{bmatrix}
-3u{^{2}}^{2} \chi +9u^{3} \lambda  & \frac{2}{3} u{^{2}}^{3} \chi +2u^{2} \lambda \\
-2u{^{2}}^{4} \chi +6u{^{2}}^{2} \lambda  & 3u{^{2}}^{2} u^{3}\chi -9u^{3} \lambda 
\end{bmatrix} ,$$
where we have put 
$$ \chi :=\frac{\partial ^{2}(\log f)}{\partial u^3\partial u^3} ,\ \lambda :=\frac{\partial (\log f)}{\partial u^2} +3u^{3}\frac{\partial ^{2}(\log f)}{\partial u^2\partial u^3} +u^{2}\frac{\partial ^{2}(\log f)}{\partial u^2\partial u^2} .$$
In the special case when $f$ is constant the Levi-Civita connection $\nabla $ is flat.
\end{theorem}

\begin{proof}
It is useful to record the action of $ \zeta _{1} ,\zeta _{2}$ on the entries of the metric $G$ in a table:
$$\begin{array}{ c|c c c }
 & G_{11} =-u^{2} f & G_{12} =\frac{9}{2} u^{3} f & G_{22} =3u{^{2}}^{2} f\\
\hline
\zeta _{1} =2u^{2}\frac{\partial }{\partial u^{2}} +3u^{3}\frac{\partial }{\partial u^{3}} & -2u^{2} f-u^{2} \zeta _{1}( f) & \frac{27}{2} u^{3} f+\frac{9}{2} u^{3} \zeta _{1}( f) & 12u{^{2}}^{2} f+3u{^{2}}^{2} \zeta _{1}( f)\\
\zeta _{2} =-9u^{3}\frac{\partial }{\partial u^{2}} +2u{^{2}}^{2}\frac{\partial }{\partial u^{3}} \  & 9u^{3} f-u^{2} \zeta _{2}( f) & 9u{^{2}}^{2} f+\frac{9}{2} u^{3} \zeta _{2}( f) & -54u^{2} u^{3} f+3u{^{2}}^{2} \zeta _{2}( f)
\end{array}$$
Using the Equations \eqref{eq:Koszulsimple} we obtain
\begin{align*}
2G\Gamma _{1} &=\begin{bmatrix}
-2u^{2} f-u^{2} \zeta _{1}( f) & 18u^{3} f-u^{2} \zeta _{2}( f)\\
9u^{3} f+9u^{3} \zeta _{1}( f) +u^{2} \zeta _{2}( f) & 12u{^{2}}^{2} f+3u{^{2}}^{2} \zeta _{1}( f)
\end{bmatrix},\\
2G\Gamma _{2} &=\begin{bmatrix}
9u^{3} f-u^{2} \zeta _{2}( f) & 12u{^{2}}^{2} f-3u{^{2}}^{2} \zeta _{1}( f) +9u^{3} \zeta _{2}( f)\\
6u{^{2}}^{2} f+3u{^{2}}^{2} \zeta _{1}( f) & -54u^{2} u^{3} f+3u{^{2}}^{2} \zeta _{2}( f)
\end{bmatrix}.
\end{align*}
To the right hand sides we apply $\frac{G^{-1}}{2} =\frac{-1}{\frac{3}{2} f^{2} \Delta }\begin{bmatrix}
3u{^{2}}^{2} f & -\frac{9}{2} u^{3} f\\
-\frac{9}{2} u^{3} f & -u^{2} f
\end{bmatrix}$ and obtain after a tedious calculation
\begin{align*}
&\frac{-1}{\frac{3}{2} f^{2} \Delta }\begin{bmatrix}
3u{^{2}}^{2} f & -\frac{9}{2} u^{3} f\\
-\frac{9}{2} u^{3} f & -u^{2} f\ 
\end{bmatrix}\begin{bmatrix}
-2u^{2} f-u^{2} \zeta _{1}( f) & 18u^{3} f-u^{2} \zeta _{2}( f)\\
9u^{3} f+9u^{3} \zeta _{1}( f) +u^{2} \zeta _{2}( f) & 12u{^{2}}^{2} f+3u{^{2}}^{2} \zeta _{1}( f)
\end{bmatrix}\\
&=\frac{-1}{\frac{3}{2} f^{2} \Delta }\begin{bmatrix}
-\frac{3\Delta }{2} f^{2} -\frac{3\Delta }{2} f\zeta _{1}( f) +\frac{3\Delta }{2} u^{2}\frac{\partial (\log f)}{\partial u^{2}} & -\frac{3\Delta }{2} u{^{2}}^{2} f\frac{\partial f}{\partial u^{3}}\\
-\frac{3\Delta }{2}\frac{1}{3} u^{2} f\frac{\partial f}{\partial u^{3}} & -3\Delta f^{2} -\frac{3\Delta }{2} u^{2} f\frac{\partial f}{\partial u^{2}}
\end{bmatrix}\\
&=\begin{bmatrix}
1+\zeta _{1}(\log f) -u^{2}\frac{\partial (\log f)}{\partial u^{2}} & u{^{2}}^{2}\frac{\partial (\log f)}{\partial u^{3}}\\
\frac{1}{3} u^{2}\frac{\partial (\log f)}{\partial u^{3}} & 2 +u^{2}\frac{\partial (\log f)}{\partial u^{2}}
\end{bmatrix}
\end{align*}
and
\begin{align*}
&\frac{-1}{\frac{3}{2} f^{2} \Delta }\begin{bmatrix}
3u{^{2}}^{2} f & -\frac{9}{2} u^{3} f\\
-\frac{9}{2} u^{3} f & -u^{2} f\ 
\end{bmatrix}\begin{bmatrix}
9u^{3} f-u^{2} \zeta _{2}( f) & 12u{^{2}}^{2} f-3u{^{2}}^{2} \zeta _{1}( f) +9u^{3} \zeta _{2}( f)\\
6u{^{2}}^{2} f+3u{^{2}}^{2} \zeta _{1}( f) & -54u^{2} u^{3} f+3u{^{2}}^{2} \zeta _{2}( f)
\end{bmatrix}\\
&=\frac{-1}{\frac{3}{2} f^{2} \Delta }\begin{bmatrix}
-\frac{3\Delta }{2} u{^{2}}^{2} f\frac{\partial f}{\partial u^{3}} & -\frac{3\Delta }{2} u^{2}\left( -6+3fu^2\frac{\partial f}{\partial u^{2}}\right)\\
-\frac{3\Delta }{2} f^{2} -\frac{3\Delta }{2} u^{2} f\frac{\partial f}{\partial u^{2}} & -\frac{3\Delta }{2}\left(9u^{3} f\frac{\partial f}{\partial u^{2}} - u{^{2}}^{2} f\frac{\partial f}{\partial u^{3}}\right)
\end{bmatrix}\\
&=\begin{bmatrix}
u{^{2}}^{2}\frac{\partial (\log f)}{\partial u^{3}} & 3u^{2}\left( -2+u^2\frac{\partial (\log f)}{\partial u^{2}}\right)\\
1+u^{2}\frac{\partial (\log f)}{\partial u^{2}} & -9u^{3}\frac{\partial (\log f)}{\partial u^{2}} + u{^{2}}^{2}\frac{\partial (\log f)}{\partial u^{3}}
\end{bmatrix} .
\end{align*}
What is remarkable, of course, is that the discriminant $\Delta$ cancels out. For the evaluation of $\mathcal{R}$ we used mathematica \cite{Mathematica}.
\end{proof}

We mention that for $f=1$ a Saito type metric $\eta:=\partial G/\partial u^3$ can be deduced from the metric $G$ of Proposition \ref{prop:S3metric}:
$$\eta=\begin{bmatrix}
    0&9/2\\
    9/2&0
\end{bmatrix}.$$
The metrics $G$ and $\eta$ form a \emph{flat pencil}, a structure that is important in Dubrovin's construction \cite{Dubrovin}. 

\section{Outlook}\label{sec:outlook}

In view of \cite{Michor} we expect that regular Levi-Civita connections exist for orbit spaces (seen as differential spaces in the sense of Sikorski) of compact groups as well as for categorical quotients of reductive groups. The main obstacle seems to be that the epimorphism $\lambda$ of Equation \eqref{eq:lambda} (or $\mathcal{C}^\infty$-versions thereof \cite{liftingHomo}) has a nontrivial kernel when the group is non-discrete. It consists of all the derivations that annihilate all the invariants. This includes the derivations tangent to the orbits.
At the moment we cannot offer an argument why this kernel should be preserved by the Levi-Civita connection. We also expect that Sjamaar-Lerman symplectic quotients \cite{SL,HSScompositio} of unitary representations possess canonical stratified Kähler structures (in the sense of \cite{HO}) whose associated metrics admit regular Levi-Civita connections.

In \cite{HOS} we studied symplectic forms on singular spaces (i.e., singular algebraic varieties or differential spaces). On the formal level, it is straight forward to extend the theory of symplectic connections \cite{Hess} to the singular case. However, when attempting to solve the symplectic analogue of the Koszul equations for the double cone it turns out to be necessary to localize along the defining equation. For this reason, we refrain from presenting the details here.

It should be said that when $\nabla$ is a Levi-Civita connection for $(\operatorname{Der}(A),A)$ and $A$ is a singular noetherian $\boldsymbol{k}$-algebra then, even though we have a curvature endomorphism,
we cannot make sense of Ricci curvature or scalar curvature. This is because in this situation $\operatorname{Der}(A)$ is not projective and therefore one cannot take traces. For this reason, we suggest when $V$ is merely a coherent $A$-module to search for $A$-linear maps $\operatorname{End}_A(V)\to A$ that are invariant under conjugation with elements of $\operatorname{Aut}_A(V)$ since they may substitute the trace. We hope that homological algebra can provide appropriate answers. Of course this is also relevant when one tries to find characteristic classes for non-projective $A$-modules $V$. If the Levi-Civita connection arises by restriction from a projective module to a non-projective submodule (as in Section \ref{sec:LC}) one has of course the option to restrict Ricci and scalar curvatures to this submodule. However, it is unclear if these quantities have an intrinsic meaning.

What is meant in physics by a space-time singularity is that the metric has poles while there is a regular local coordinate system for the space-time. In the examples considered in Section \ref{sec:LC} the situation is opposite. It seems to us an interesting question if the two notions can be related to each other via a resolution of singularities. In other words, does there exist a (possibly singular) metric on a crepant resolution of a singularity of a semi-Riemannian orbifold such that the respective Levi-Civita connections are mapped into each other when restricted to the principal stratum? Note, however, that the theory of resolutions of singularities has not yet been elaborated for semialgebraic sets, seen as differential spaces.

Let us indicate further directions in which our investigations may be continued.
Inspecting the differential geometry literature, there might be of course other types of characteristic classes that turn out to be  generalizable to the singular case. Moreover, if the base algebra $A$ permits to solve ordinary differential equations (e.g., in the case of differential spaces or in a holomorphic setup) also the question of stratified holonomy respresentations may be worth studying.
In mathematical newspeak the subject of our paper may be referred to as \emph{prederived geometry}\footnote{It is, of course, just geometry.}.
It seems to be an interesting question if it is possible to `derive' it. In other words, one might be able to lift an $L$-connection on $V$ to a homotopy connection on a free resolution of $V$ for the $L_\infty$-algebroid on a free resolution of $L$ that was constructed in \cite{LGR,HO}.

\section{Conflicts of interest}
On behalf of all authors, the corresponding author states that there is no conflict of interest. 
 
\bibliographystyle{amsplain}
\bibliography{explorations.bib}
\end{document}